\documentclass[11pt,a4paper,leqno]{amsart}

\title
[Wigner distribution of Gaussian tempered $\gsp$s]
{The Wigner distribution of Gaussian tempered generalized stochastic processes}

\author[P. Wahlberg]{Patrik Wahlberg}

\address{Dipartimento di Scienze Matematiche, Politecnico di Torino, Corso Duca degli Abruzzi 24,
10129 Torino, Italy}

\email{patrik.wahlberg[AT]polito.it}

\usepackage{latexsym}
\usepackage[a4paper]{geometry}
\usepackage{amsmath}
\usepackage{amssymb}
\usepackage{amsthm}
\usepackage{bbm}
\usepackage{amsfonts}
\usepackage{mathrsfs}
\usepackage{calc}
\usepackage{cite}
\usepackage{color}
\usepackage{graphicx}
\usepackage{enumerate}

\numberwithin{equation}{section}          

\newtheorem{thm}{Theorem}
\numberwithin{thm}{section}

\newcommand{\rubrik}{}
\newtheorem{prop}[thm]{Proposition}
\newtheorem{cor}[thm]{Corollary}
\newtheorem{lem}[thm]{Lemma}

\theoremstyle{definition}

\newtheorem{defn}[thm]{Definition}
\newtheorem{example}[thm]{Example}

\theoremstyle{remark}

\newtheorem{rem}[thm]{Remark}              

\newcommand{\scal}[2]{\langle #1,#2\rangle}
\newcommand{\pd}[1] {\partial ^#1}

\newcommand{\bP}{\mathbf P}
\newcommand{\bE}{\mathbf E}
\newcommand{\ro}{\mathbf R}
\newcommand{\no}{\mathbf N}
\newcommand{\rr}[1]{\mathbf R^{#1}}

\newcommand{\sro}[1]{\mathbf S}
\newcommand{\nn}[1]{\mathbf N^{#1}}

\newcommand{\co}{\mathbf C}
\newcommand{\cc}[1]{\mathbf C^{#1}}

\newcommand{\dd}{\mathrm {d}}

\newcommand{\fy}{\varphi}

\newcommand{\cdo}{\, \cdot \, }

\newcommand{\supp}{\operatorname{supp}}

\newcommand{\eabs}[1]{\langle #1\rangle}

\newcommand{\GL}{\operatorname{GL}}

\newcommand{\cS}{\mathscr{S}}

\newcommand{\cB}{\mathscr{B}}
\newcommand{\cD}{\mathscr{D}}
\newcommand{\cF}{\mathscr{F}}
\newcommand{\kF}{\mathcal{F}}

\newcommand{\cK}{\mathscr{K}}
\newcommand{\cL}{\mathscr{L}}

\newcommand{\J}{\mathcal{J}}
\newcommand{\wt}{\widetilde}
\newcommand{\wh}{\widehat}
\newcommand{\re}{{\rm Re}}
\newcommand{\im}{{\rm Im}}

\def\la{\langle}
\def\ra{\rangle}
\newcommand{\leqs}{\leqslant}
\newcommand{\geqs}{\geqslant}
\newcommand{\gsp}{\operatorname{GSP}}



\begin{document}

\maketitle

\begin{abstract}
We define the Wigner distribution of a tempered generalized stochastic process
that is complex-valued symmetric Gaussian. This gives a time-frequency generalized stochastic process 
defined on the phase space. 
We study its covariance and our main result is a formula for the Weyl symbol
of the covariance operator, expressed in terms of the Weyl symbol of the covariance operator of the 
original generalized stochastic process. 
\end{abstract}

\par

\section{Introduction}\label{sec:intro}

The Wigner distribution is a fundamental concept in quantum mechanics, signal analysis \cite{Flandrin2,Grochenig1} and linear partial differential equations, 
where it appears in the Weyl calculus of pseudodifferential operators \cite{Folland1,Hormander1}.  
It was introduced 1932 in the infancy of mathematical quantum mechanics by Wigner \cite{Wigner1}. 
There it serves as a candidate for a probability density function for a particle in phase space, 
an endeavour that may be seen to fail due to its lack of non-negativity, 
in certain interesting ways related to the uncertainty principle.
In signal analysis the Wigner distribution is used as a time-frequency distribution with high resolution 
but subject to the same interpretational issues as in quantum mechanics. 
 
In the Weyl calculus of pseudodifferential operators the cross-Wigner distribution 
of two functions $f,g$ on $\rr d$
\begin{equation*}
W(g,f) (x,\xi) 
= (2 \pi)^{-\frac{d}{2}} \int_{\rr d} g (x+y/2) \overline{f(x-y/2)} e^{- i \la y, \xi \ra} \dd y 
\end{equation*}
appears in the formula $( a^w(x,D) f, g) = (2 \pi)^{-\frac{d}{2}} ( a, W(g,f) )$
which connects a Weyl pseudodifferential operator $a^w(x,D)$ and its symbol $a$
which is a function defined on phase space $T^* \rr d$. 

By now many properties of the Wigner distribution $W(f) = W(f,f)$ are known for functions and distributions $f$. 
In signal analysis much work has been devoted to the Wigner distribution of stochastic processes \cite{Boashash1,Flandrin1,Hlawatsch1,Martin1}. 
The analysis of the Wigner distribution of generalized stochastic processes dates back to the PhD thesis of A.~J.~E.~M.~Janssen \cite{Janssen1}.
He studied the expected value of this Wigner distribution, also known in the applied literature as the Wigner spectrum.
Many authors have contributed to this study \cite{Boashash1,Feichtinger1,Feichtinger2,Flandrin2,Flandrin3,Hlawatsch1,Keville1}, 
already in the 1990s, and the activity has since then abated. 

Nevertheless it seems that not much attention has been devoted to the second order statistical properties of the Wigner distribution of 
generalized stochastic processes, except for \cite[Chapter~9.1]{Boashash1} by Stankovi\'c, and \cite{Wahlberg1} which may be seen as a particular case of some of the results presented here. 
We mean the study of its covariance function or distribution. 
This apparent gap in the literature is the motivation for the present work. 

The main new feature of this paper compared to \cite{Boashash1,Feichtinger2,Flandrin2,Hormann1,Janssen1,Keville1,Martin1} is hence the study of the covariance properties 
of the Wigner distribution of generalized stochastic processes. 
The added assumption that the generalized stochastic processes be Gaussian and symmetric allows 
such an analysis. 

A covariance distribution is the Schwartz kernel of a linear covariance operator. 
This admits the transfer of the study of covariance properties to the study of operators, or equivalently 
to the study of the corresponding Weyl symbols in the framework of pseudodifferential operators. 
It turns out to be convenient to work on the Weyl symbol side. 

Our main result is the formula for $x_1, x_2, \xi_1, \xi_2 \in \rr d$
\begin{equation*}
\sigma_W (x_1, x_2, \xi_1, \xi_2) = 
\sigma_u \left( x_1 - \frac12 \xi_2, x_2 + \frac12 \xi_1 \right) \sigma_u \left( x_1 + \frac12 \xi_2, x_2 - \frac12 \xi_1 \right) 
\end{equation*}
which expresses the Weyl symbol $\sigma_W$ of the zero mean Wigner distribution (where we subtract the mean from the Wigner distribution) of a Gaussian symmetric tempered generalized stochastic process $u$, in terms of the Weyl symbol $\sigma_u$ (a k a the Wigner spectrum) of the covariance operator for $u$. 

White noise is characterized as a generalized stochastic process $u$ with $\sigma_u$ equal to a positive constant. 
Thus the formula above says in particular that Gaussian symmetric white noise on $\rr d$ has a zero mean Wigner distribution which is white noise on the phase space $T^* \rr d$. 

We work out some consequences of the formula when we add the assumption on $u$ to be stationary. 
It then turns out that the covariance operator for the Wigner distribution of $u$, that is the Weyl quantization of $\sigma_W$, 
commutes with translation of the first (``time'') variable, and also with modulation in the second (``frequency'') variable. 
These observations support the interpretation of the Wigner distribution as a time-frequency representation 
of generalized stochastic processes. 

We note that Flandrin \cite{Flandrin3} has studied the covariance properties of the spectrogram with a Gaussian window function of white symmetric Gaussian noise. 
In \cite{Abreu1,Bardenet1} an observation by Flandrin concerning the zeros of the spectrogram of white noise is analyzed, 
and \cite{Abreu2} treats analytic continuation of Brownian motion.
These works are recent contributions in the vicinity of the present paper albeit not directly connected. 

Our analysis is based on Gelfand's and Vilenkin's concept of generalized stochastic process
as a linear continuous map from a space of test functions into a space of random variables \cite[Chapter~3]{Gelfand4}. 
A more recent and nowadays established framework for generalized stochastic processes is 
the white noise analysis developed by T.~Hida starting in 1975 \cite{Biagini1,Kuo1}, 
which is an infinite-dimensional stochastic calculus with many ramifications and applications. 
This framework admits the concept of white noise and it is used in \cite{Bardenet1,Abreu1,Abreu2}.

White noise can also be defined in the Gelfand--Vilenkin framework as 
a generalized stochastic process whose covariance operator is the identity times a positive constant. 
For the purposes of this article it suffices to work in the latter framework. 
Thus we eschew the full stochastic white noise calculus, and leave it
as an open problem to formulate the Wigner distribution in this framework for the future.

The paper is organized in the following way. 
Section \ref{sec:prelim} contains notations and conventions. 
In Section \ref{sec:GSPWeylWigner} we define stochastic processes, generalized stochastic processes,
the subspaces of tempered and stationary generalized stochastic processes, and white noise. 
We also give the necessary background on Weyl pseudodifferential operators and their connection to the Wigner distribution. 

Section \ref{sec:Gaussiantempered} introduces our framework of 
Gaussian symmetric tempered generalized stochastic processes.
In Section \ref{sec:wignergeneral} we define 
the Wigner distribution of such generalized stochastic processes, 
and we prove our main result Theorem \ref{thm:symbol}. 

Consequences of this result for stationary, and frequency stationary, generalized stochastic processes
are discussed in Section \ref{sec:wignerstationary}, in particular we look at white noise. 
Section \ref{sec:WignerBrown} is devoted to Brownian motion. 
We rededuce a formula by Flandrin \cite{Flandrin2} for its Wigner spectrum, and we apply Theorem \ref{thm:symbol}
which gives a formula for the corresponding $\sigma_W$. 
In Section \ref{sec:NonnegWeylCalc} we deduce a consequence of Theorem \ref{thm:symbol} concerning non-negative pseudodifferential operators. 
Finally in Section \ref{sec:NonnegWignerSpectrum} we treat briefly (without saying something new) Flandrin's observation \cite{Flandrin1} 
that Wigner spectra are much more often non-negative than Wigner distributions of deterministic functions.

\section{Preliminaries}\label{sec:prelim}

The notation $K \Subset \rr d$ means that $K$ is a compact subset of $\rr d$, 
and we use $\ro_+$ to denote the non-negative real numbers. 
The partial derivative $D_j = - i \partial_j$, $1 \leqs j \leqs d$, acts on functions and distributions on $\rr d$, 
with extension to multi-indices as $D^\alpha = i^{-|\alpha|} \partial^\alpha$ for $\alpha \in \nn d$. 
We use the bracket $\eabs{x} = (1 + |x|^2)^{\frac12}$ for $x \in \rr d$. 
The phase or time-frequency space of $\rr d$ is $T^* \rr d = \rr {2d}$. 

The symbol $C_c^\infty(\rr d)$ denotes the compactly supported smooth test functions, and 
$\cD'(\rr d)$ its topological dual, the space of distributions. 
The Schwartz space $\cS(\rr d)$ consists of smooth functions for which the derivative of any order 
has superpolynomial decay at infinity. 
Thus $C_c^\infty(\rr d) \subseteq \cS(\rr d)$ and the dual space to $\cS(\rr d)$, called the space of tempered distributions, satisfies $\cS'(\rr d) \subseteq \cD' (\rr d)$. 
If $H$ is a Hilbert space then $L^1_{\rm loc} (\rr d, H)$ denotes the space of strongly measurable locally integrable functions with values in $H$. 
The space of linear continuous operators from a topological vector space $X$ to another such space $Y$ is denoted $\cL(X,Y)$. 
The same notation is used if the operator is conjugate linear.

The normalization of the Fourier transform is
\begin{equation*}
 \cF f (\xi )= \widehat f(\xi ) = (2\pi )^{-\frac d2} \int _{\rr
{d}} f(x)e^{-i\scal  x\xi }\, \dd x, \qquad \xi \in \rr d, 
\end{equation*}
for $f\in \cS(\rr d)$, where $\scal \cdo \cdo$ denotes the scalar product on $\rr d$. 
Then we have for $f, g \in \cS(\rr d)$
\begin{equation}\label{eq:convolutionFourier}
\wh{f * g} = (2\pi )^{\frac d2} \wh f \, \wh g
\end{equation}
and this identity extends to $f \in \cS'(\rr d)$ and $g \in \cS(\rr d)$,
with the Fourier transform defined on $\cS'(\rr d)$ 
as $( \wh f, \wh g) = (f,g)$ for $f \in \cS'(\rr d)$ and $g \in \cS(\rr d)$. 

We denote by $\cF_2$ the partial Fourier transform with respect to the second $d$ coordinates in $\rr {2d}$ of a tempered distribution defined on $\rr {2d}$, 
or the partial Fourier transform with respect to the second $2d$ coordinates in $\rr {4d}$ of a tempered distribution defined on $\rr {4d}$.
The distinction will be clear from the context. 
Finally $\kF_j$, $j = 1,2,3,4$, denotes the partial Fourier transform with respect to the $j$th $d$ coordinates in $\rr {4d}$ of 
a tempered distribution defined on $\rr {4d}$. 
Thus we may write coherently $\cF_2 = \kF_3 \kF_4$ on $\rr {4d}$.

The conjugate linear (antilinear) action of a distribution $u$ on a test function $\phi$ is written $(u,\phi)$, consistent with the $L^2$ inner product $(\cdo ,\cdo ) = (\cdo ,\cdo )_{L^2}$ which is conjugate linear in the second argument. 
Translation of a function or a distribution $f$ is denoted $T_x f(y) = f(y-x)$ for $x,y \in \rr d$, 
and modulation as $M_\xi f(x) = e^{i \la x, \xi \ra} f(x)$ for $x,\xi \in \rr d$.

\section{Generalized stochastic processes, Weyl pseudodifferential operators and the Wigner distribution}\label{sec:GSPWeylWigner}

Let $\Omega$ be a sample space equipped with a $\sigma$-algebra $\cB$ of subsets of $\Omega$
and let $\bP$ be a probability measure defined on $\cB$. 
The space of $\co$-valued second order random variables is the Hilbert space $L^2(\Omega)$ equipped with the inner product
\begin{equation*}
L^2(\Omega) \times L^2(\Omega) \ni (X,Y) \mapsto \bE ( X \overline Y ) = (X, Y )_{L^2(\Omega)}
\end{equation*}
where 
\begin{equation*}
\bE X = \int_{\Omega} X(\omega) \bP( \dd \omega)
\end{equation*}
is the expectation functional (integral). 
The Hilbert subspace of $L^2(\Omega)$ of zero mean random variables is denoted $L_0^2(\Omega)$, and each element $X \in L_0^2(\Omega)$
thus satisfies $\bE X = 0$ and $\bE |X|^2 < \infty$.

\subsection{Second order stochastic processes}\label{subsec:stochproc}

A second order stochastic process is an element $f \in L_{\rm{loc}}^1( \rr d, L^2(\Omega) )$.
The cross-covariance function of $f,g \in L_{\rm{loc}}^1 ( \rr d, L^2(\Omega) )$ is
\begin{equation*}
k_{fg} ( x ,y) = \bE ( f(x)  \overline{g(y)} ), \quad x, y \in \rr d. 
\end{equation*}
The function $k_f = k_{f f}$ is called (auto-)covariance function of $f$. 
By the Cauchy--Schwarz inequality in $L^2(\Omega)$ we have 
\begin{equation*}
|k_{f g} ( x ,y)|^2 
\leqs k_f ( x ,x) \, k_g ( y ,y), \quad x, y \in \rr d,
\end{equation*}
which implies that $k_{f g}$ extends to an element in $\cD'(\rr {2d})$ \cite{Gelfand4} when $f,g \in L_{\rm{loc}}^1( \rr d, L^2(\Omega) )$. 
Defining the cross-covariance operator as
\begin{equation*}
( \cK_{f g} \fy, \psi) = ( k_{f g}, \psi \otimes \overline \fy )_{L^2(\rr {2d} )}, \quad \fy, \psi \in C_c^\infty(\rr d), 
\end{equation*}
yields a continuous linear cross-covariance operator $\cK_{f g}: C_c^\infty(\rr d) \to \cD'(\rr d)$. 
This is a consequence of Schwartz's kernel theorem \cite[Theorem~5.2.1]{Hormander1}. 

Since Fubini's theorem gives
\begin{align*}
( k_f, \fy \otimes \overline \fy )_{L^2(\rr {2d})}
& = \bE \left| \int_{\rr d}   f(x) \overline{\fy (x)} \, \dd x \right|^2
\geqs 0 \quad \forall \fy \in C_c^\infty(\rr d),    
\end{align*}
it is clear that $k_f$ is the kernel of a non-negative linear continuous covariance operator $\cK_f = \cK_{f f}: C_c^\infty(\rr d) \to \cD'(\rr d)$.

\subsection{Generalized stochastic processes}\label{subsec:GSP}

For brevity we drop the term second order which is understood in the sequel. 
In this article we will use the following definition of generalized stochastic processes due to Gelfand and Vilenkin \cite{Gelfand4}. 

\begin{defn}\label{def:gsp}
A generalized stochastic process ($\gsp$) $u \in \cL ( C_c^\infty(\rr d), L^2(\Omega) )$ is a 
conjugate linear continuous operator $u: C_c^\infty(\rr d) \to L^2(\Omega)$, 
written as $(u,\fy) \in L^2(\Omega)$ for $\fy \in C_c^\infty(\rr d)$. 
A zero mean GSP takes values in $L_0^2(\Omega)$, that is 
$u \in \cL ( C_c^\infty(\rr d), L_0^2(\Omega) )$, and has acronym $\gsp_0$.
\end{defn}

Thus the space of $\gsp$s $\cL ( C_c^\infty(\rr d), L^2(\Omega) )$ is a conjugate linear space of continuous operators 
from the locally convex topological vector space $C_c^\infty(\rr d)$ into the Hilbert space $L^2(\Omega)$.
For each $K \Subset \rr d$ there exist $C > 0$ and $k \in \no$ 
such that  
\begin{equation*}
\| (u,\fy) \|_{L^2(\Omega)}
\leqs C \sum_{|\alpha| \leqs k} \sup_{x \in \rr d} |\pd \alpha \fy (x)|, \quad \fy \in C_c^\infty(K).
\end{equation*}

As in ordinary distribution theory \cite{Gelfand1,Hormander1} a $\gsp$ $u$ is always differentiable as
\begin{equation*}
( D^\alpha u,\fy) = (u, D^\alpha \fy), \quad \fy \in C_c^\infty (\rr d), \quad \alpha \in \nn d. 
\end{equation*}
If $u \in  \cL ( C_c^\infty(\rr d), L^2(\Omega) )$ then 
the mean $m_u$ is defined by 
\begin{equation}\label{eq:meandist}
(m_u, \fy  ) 
= \bE (u,\fy), \quad \fy \in C_c^\infty (\rr d). 
\end{equation}
Due to $\bE |X| \leqs \left( \bE |X|^2 \right)^{\frac12}$, that is the embedding $L^2(\Omega) \subseteq L^1(\Omega)$,
we have $m_u \in \cD'(\rr d)$.

Let $u,v \in  \cL ( C_c^\infty(\rr d), L^2(\Omega) )$. The cross-covariance distribution $k_{u v}$ is defined by 
\begin{equation}\label{eq:crosscovdist}
(k_{u v}, \fy \otimes \overline \psi ) 
= \bE ( (u,\fy)  \overline{ (v,\psi)} ), \quad \fy, \psi \in C_c^\infty (\rr d). 
\end{equation}
It satisfies $k_{u v}(x,y) = \overline{ k_{v u}(y,x)}$, and for each pair $K_1, K_2 \Subset \rr d$
there is $C > 0$ and $k_1, k_2 \in \no$ such that 
\begin{align*}
|(k_{u v}, \fy \otimes \overline \psi )|
& \leqs C \sum_{|\alpha| \leqs k_1} \sup_{x \in \rr d} |\pd \alpha \fy (x)| \, 
\sum_{|\beta| \leqs k_2} \sup_{x \in \rr d} |\pd \beta \psi (x)|, \\
& \qquad \qquad \fy \in C_c^\infty(K_1), \quad \psi \in C_c^\infty(K_2). 
\end{align*}
Thus $k_{u v}$ is a sesquilinear continuous form on $C_c^\infty(\rr d) \times C_c^\infty(\rr d)$. 
If we equip $\cD'(\rr d)$ with its weak$^*$ topology then 
\begin{equation*}
( \cK_{u v} \psi, \fy ) = (k_{u v}, \fy \otimes \overline \psi ), \quad \fy, \psi \in C_c^\infty(\rr d), 
\end{equation*}
defines a continuous linear operator $\cK_{u v}: C_c^\infty(\rr d) \to \cD' (\rr d)$, 
called the cross-covariance operator. 
From the Schwartz kernel theorem \cite[Theorem~5.2.1]{Hormander1} it follows that $k_{u v} \in \cD'(\rr {2d})$. 

If $v = u$ then we call $k_u = k_{u u} \in \cD'(\rr {2d})$ the (auto-)covariance distribution
and $\cK_u = \cK_{u u} \in \cL( C_c^\infty(\rr d), \cD' (\rr d) )$ the (auto-)covariance operator of $u$. 
Then $(k_u, \fy \otimes \overline \fy ) \geqs 0$ for all $\fy \in C_c^\infty(\rr d)$
so $\cK_u \geqs 0$ is non-negative in the sense of 
\begin{equation*}
( \cK_u \fy, \fy ) \geqs 0 \quad \forall \fy \in C_c^\infty(\rr d). 
\end{equation*}
It holds $( \cK_u \fy, \psi ) = \overline{( \cK_u \psi, \fy )}$ for all $\fy,\psi \in C_c^\infty(\rr d)$. 

A stochastic process $f \in L_{\rm{loc}}^1(\rr d, L^2(\Omega) )$ may be considered a generalized stochastic process in $\cL ( C_c^\infty(\rr d) , L^2(\Omega) )$ by means of 
\begin{equation*}
C_c^\infty (\rr d) \ni \fy \mapsto (f, \fy) = \int_{\rr d} f(x ) \, \overline{ \fy (x) } \dd x, 
\end{equation*}
and then there is consistency between the covariance function and the covariance distribution, in the sense of 
\begin{equation*}
( k_f, \psi \otimes \overline \fy )_{L^2(\rr {2d})}
= \iint_{\rr {2d}} \bE \left( f(x) \overline{ f(y)}  \right) 
\overline{\psi(x)} \fy (y) \dd x \, \dd y
= \bE \left( (f, \psi) \overline{ (f, \fy)} \right). 
\end{equation*}
%

\subsection{Tempered generalized stochastic processes}\label{subsec:temperedGSP}

In the sequel we will work with tempered generalized stochastic processes. 
This means that we extend the domain in Definition \ref{def:gsp}, that is the test function space $C_c^\infty(\rr d)$, 
into the Schwartz space $\cS(\rr d) \supseteq C_c^\infty(\rr d)$. 
It thus leads to a smaller space of generalized stochastic processes 
$\cL ( \cS(\rr d), L^2(\Omega) ) \subseteq \cL ( C_c^\infty(\rr d), L^2(\Omega) )$. 

\begin{defn}\label{def:gsptemp}
A \emph{tempered} $\gsp$ $u \in \cL ( \cS(\rr d), L^2(\Omega) )$ is a 
conjugate linear continuous operator $u: \cS(\rr d) \to L^2(\Omega)$, 
written as $(u,\fy) \in L^2(\Omega)$ for $\fy \in \cS(\rr d)$. 
\end{defn}

If $u,v$ are tempered $\gsp$s then $k_{u v} \in \cS'(\rr {2d})$,
the cross-covariance operator is continuous
$\cK_{u v}: \cS(\rr d) \to \cS' (\rr d)$, 
and $\cK_u \geqs 0$ on $\cS(\rr d)$.
For a tempered $\gsp$ $u \in \cL ( \cS(\rr d), L^2(\Omega) )$
the Fourier transform can be defined as $(\wh u, \wh \fy) = (u, \fy)$ for $\fy \in \cS(\rr d)$,
as in ordinary distribution theory \cite{Hormander1}.
The following simple device will be useful. 

\begin{lem}\label{lem:FourierCovOp}
If $u \in \cL ( \cS(\rr d), L^2(\Omega) )$ then the covariance operators of $u$ and $\wh u$
are related as $\cK_{\wh u} = \cF \cK_u \cF^{-1}$. 
\end{lem}

\begin{proof}
For $\fy, \psi \in \cS(\rr d)$ we have 
\begin{align*}
( \cK_{\wh u} \psi, \fy )  
& = (k_{\wh u}, \fy \otimes \overline \psi ) 
= \bE ( (\wh u,\fy)  \overline{ (\wh u,\psi)} )
= \bE ( (u, \cF^{-1} \fy)  \overline{ (u, \cF^{-1} \psi)} ) \\
& = (k_{u}, \cF^{-1} \fy \otimes \overline{\cF^{-1} \psi} )
= ( \cK_u \cF^{-1} \psi, \cF^{-1} \fy)
= ( \cF \cK_u \cF^{-1} \psi, \fy). 
\end{align*}
\end{proof}

\subsection{Stationary generalized stochastic processes}\label{subsec:stat}

A very well studied subspace of $\cL ( C_c^\infty(\rr d) , L^2(\Omega) )$
is the space of stationary $\gsp$s.

\begin{defn}\label{def:stat}
Let $u \in \cL ( C_c^\infty(\rr d) , L^2(\Omega) )$ be a $\gsp$ 
so that $m_u \in \cD'(\rr d)$ (cf. Eq. \eqref{eq:meandist}) and $k_u \in \cD'(\rr {2d})$.  
Then $u$ is said to be \emph{stationary} if $m_u = c \in \co$ is constant, and its covariance distribution $k_u$ is translation invariant as
\begin{equation*}
(k_u, T_x \fy \otimes T_x \overline{\psi}) = (k_u, \fy \otimes \overline{\psi}) \quad \forall \fy, \psi \in C_c^\infty (\rr d) \quad \forall x \in \rr d, 
\end{equation*}
cf. \cite[Chapter 3, \S3]{Gelfand4}.
\end{defn}

Note that the condition of stationarity means that the covariance operator commutes with translation: 
\begin{align*}
( \cK_u T_x \psi, \fy) 
& = (k_u, T_x T_{-x} \fy \otimes T_x \overline{\psi}) 
= (k_u, T_{-x} \fy \otimes \overline{\psi}) 
= ( \cK_u \psi, T_{-x} \fy ) \\
& = ( T_x \cK_u \psi,  \fy ) \quad \forall \fy, \psi \in C_c^\infty (\rr d) \quad \forall x \in \rr d, 
\end{align*}
that is 
\begin{equation}\label{eq:TranslationInvariance}
\cK_u T_x  = T_x \cK_u \quad \forall x \in \rr d.
\end{equation}

\begin{rem}\label{rem:stationary}
In the literature \cite{Doob1,Miller1} Definition \ref{def:stat} is usually designated \emph{wide-sense} or \emph{weakly stationary}, 
and the term \emph{stationary} is used if 
for any $n \in \no$, any $\{ \fy_j \}_{j=1}^{n} \subseteq C_c^\infty (\rr d)$, and any $x \in \rr d$, 
the two sets of random variables
\begin{equation*}
\{ (u, \fy_1), \cdots ,(u, \fy_n) \}
\end{equation*}
and 
\begin{equation*}
\{ (u, T_x \fy_1), \cdots ,(u, T_x \fy_n) \}
\end{equation*}
have identical probability laws. 
Stationarity is a stronger property than wide-sense stationarity in general. 
For Gaussian $\gsp$s the two concepts coincide \cite{Doob1}. 
In this paper we study Gaussian $\gsp$s and use only the term stationary. 
\end{rem}

If a $\gsp$ $u \in \cL ( C_c^\infty(\rr d) , L^2(\Omega) )$ is stationary then there exists
$h_u \in \cD'(\rr d)$ such that 
\begin{equation*}
(k_u, \fy \otimes \overline{\psi})
= (h_u, \fy * \psi^* ) \quad \forall \fy, \psi \in C_c^\infty (\rr d), 
\end{equation*}
where $\psi^*(x) = \overline{\psi (-x)}$, 
see \cite[Chapter 2, \S3.5 and Chapter 3, \S3.2]{Gelfand4} and \cite[Theorem~3.1.4$'$]{Hormander1}. 
Equivalently we have 
\begin{equation}\label{eq:kernelstat}
k_u = (1 \otimes h_u) \circ \kappa^{-1}
\end{equation}
where $\kappa$ denotes the matrix 
\begin{equation}\label{eq:kappa}
\kappa = 
\left(
\begin{array}{ll}
I_d & \frac12 I_d \\
I_d & - \frac12 I_d
\end{array}
\right) \in \rr {2d \times 2d}
\end{equation}
whose inverse is  
\begin{equation}\label{eq:kappainv}
\kappa^{-1} = 
\left(
\begin{array}{ll}
\frac12 I_d & \frac12 I_d \\
I_d & - I_d
\end{array}
\right) \in \rr {2d \times 2d}.
\end{equation}

We also have for $\fy, \psi \in C_c^\infty (\rr d)$
\begin{equation}\label{eq:CovOpConv}
( \cK_u \psi,\fy )
= ( k_u, \fy \otimes \overline{\psi} )
= ( h_u, \fy * \psi^* )
= \int_{\rr d} ( h_u, \overline{ \psi(x - \cdot) } ) \, \overline{ \fy(x) } \dd x
= ( h_u * \psi, \fy )
\end{equation}
which means that $\cK_u \psi = h_u * \psi$. 

By the Bochner--Schwartz theorem \cite[Chapter 2, \S3.3, Theorem~3]{Gelfand4} there exists a spectral non-negative Radon measure 
\begin{equation}\label{eq:SpectralMeasure}
\mu_u = (2 \pi)^{\frac{d}{2}} \wh h_u
\end{equation}
on $\rr d$ that satisfies 
\begin{equation}\label{eq:temperedmeasure}
\int_{\rr d} \eabs{x}^{-s}  \, \dd \mu_u (x) < \infty
\end{equation}
for some $s \geqs 0$. 
Such a measure is called tempered. 

Hence (cf. \eqref{eq:convolutionFourier})
\begin{equation*}
( h_u, \fy * \psi^* ) = \int_{\rr d}  \wh \psi (\xi) \, \overline{ \wh \fy (\xi) } \, \dd \mu_u (\xi). 
\end{equation*}
This gives for $\fy, \psi \in C_c^\infty (\rr d)$
\begin{equation}\label{eq:gspstat1}
\bE \left( (u, \fy ) \overline{(u,\psi)} \right)
= (k_u, \fy \otimes \overline{\psi})
= ( \cK_u \psi, \fy)
= ( h_u, \fy * \psi^* ) 
= ( \psi, \fy )_{\cF L^2(\mu_u)}
\end{equation}
and in particular
\begin{equation}\label{eq:gspstat2}
\| (u, \fy ) \|_{L^2(\Omega)}^2
= \int_{\rr d}  |\wh \fy (\xi) |^2 \, \dd \mu_u (\xi). 
\end{equation}
Here $\cF L^2(\mu_u) \subseteq \cS'(\rr d)$ is the Hilbert subspace of tempered distributions $f$ such that $\wh f \in L_{\rm{loc}}^2(\mu_u)$
and $\wh f$ is square integrable with respect to $\mu_u$. 

Let $u \in \cL ( C_c^\infty(\rr d) , L^2(\Omega) )$ be stationary and denote by $\mu_u$ the corresponding spectral non-negative tempered Radon measure. 
The identity \eqref{eq:gspstat1} implies that the linear map 
$\fy \mapsto \overline{(u,\fy) }$ extends uniquely to 
an isometry between Hilbert spaces $\cF L^2(\mu_u) \to L^2(\Omega)$ \cite{Conway1},
and $\cK_u$ extends uniquely to the identity operator on $\cF L^2(\mu_u)$.
Since $\cS(\rr d) \subseteq \cF L^2(\mu_u)$ it follows that a stationary $\gsp$ is always tempered.

Lemma \ref{lem:FourierCovOp} combined with $\cK_u \psi = h_u * \psi$, 
\eqref{eq:convolutionFourier}
and \eqref{eq:SpectralMeasure} give for $f \in \cS(\rr d)$
\begin{equation}\label{eq:}
\cK_{\wh u} f 
= \cF \left( h_u * \cF^{-1} f \right)
= \mu_u f
\end{equation}
which means that $\wh u$ has a multiplicative covariance operator.

The following definition appears in \cite[Definition~3]{Feichtinger2} and \cite[Definition~2.2.2]{Janssen1}.

\begin{defn}\label{def:FrequencyStat}
Let $u \in \cL ( \cS(\rr d) , L^2(\Omega) )$. 
If $\wh u \in \cL ( \cS(\rr d) , L^2(\Omega) )$ is stationary then $u$ is called frequency stationary. 
\end{defn}

From Lemma \ref{lem:FourierCovOp} and \eqref{eq:TranslationInvariance}
it follows that a frequency stationary tempered $\gsp$ $u$ has a modulation invariant
covariance operator:
\begin{equation}\label{eq:ModulationInvariance}
\cK_u M_x  = M_x \cK_u \quad \forall x \in \rr d.
\end{equation}
%

\subsection{Weyl pseudodifferential operators and the Wigner distribution}

If $\sigma \in \cS(\rr {2d})$ is a Weyl symbol then the Weyl pseudodifferential operator \cite{Folland1,Hormander1} is defined as 
\begin{equation}\label{eq:weylquantization}
\sigma^w(x,D) f(x)
= (2\pi)^{-d}  \int_{\rr {2d}} e^{i \langle x-y, \xi \rangle} \sigma \left(\frac{x+y}{2},\xi \right) \, f(y) \, \dd y \, \dd \xi, \quad f \in \cS(\rr d), 
\end{equation}
and $\sigma^w(x,D): \cS(\rr d) \to \cS(\rr d)$ continuously. 
By the invariance of $\cS'(\rr {2d})$ under linear invertible coordinate transformations and partial Fourier transforms, 
the Weyl correspondence extends to $\sigma \in \cS'(\rr {2d})$ in which case $\sigma^w(x,D): \cS(\rr d) \to \cS' (\rr d)$ is continuous. 
Note that $\sigma^w(x,D) = I$ is the identity operator if $\sigma \equiv 1$.

If $\sigma \in \cS'(\rr {2d})$ then 
\begin{equation}\label{eq:wignerweyl}
( \sigma^w(x,D) f, g) = (2 \pi)^{-\frac{d}{2}} ( \sigma, W(g,f) ), \quad f, g \in \cS(\rr d), 
\end{equation}
where the cross-Wigner distribution \cite{Folland1,Grochenig1} is defined as 
\begin{equation}\label{eq:WignerSchwartz}
\begin{aligned}
W(g,f) (x,\xi) 
& = (2 \pi)^{-\frac{d}{2}} \int_{\rr d} g (x+y/2) \overline{f(x-y/2)} e^{- i \la y, \xi \ra} \dd y \\
& = \cF_2 \left( ( g \otimes \overline f) \circ \kappa \right)(x,\xi), \quad (x,\xi) \in T^* \rr d, 
\end{aligned}
\end{equation}
where $\cF_2$ denotes the partial Fourier transformation with respect to the second $\rr d$ variable in $\rr {2d}$
and $\kappa$ is the matrix \eqref{eq:kappa}. 
We have $W(g,f) \in \cS (\rr {2d})$ when $f,g \in \cS(\rr d)$. 
The Wigner distribution of $f \in \cS(\rr d)$ is $W(f) = W(f,f)$. 

Conversely for any continuous linear operator $\cK: \cS(\rr d) \to \cS' (\rr d)$ there exists 
a kernel $k \in  \cS' (\rr {2d})$
and a Weyl symbol $\sigma_k \in  \cS' (\rr {2d})$ such that 
\begin{equation*}
( \cK f,g) = (k, g \otimes \overline f) 
= (2 \pi)^{-\frac{d}{2}} ( \sigma_k, W(g,f) ), \quad f,g \in \cS(\rr d). 
\end{equation*}
By \eqref{eq:WignerSchwartz} we have 
\begin{equation*}
( \sigma_k, W(g,f) )
= ( ( \cF_2^{-1} \sigma_k ) \circ \kappa^{-1}, g \otimes \overline f) 
\end{equation*}
which means that the kernel is related to the Weyl symbol as $k = (2 \pi)^{-\frac{d}{2}} ( \cF_2^{-1} \sigma_k ) \circ \kappa^{-1}$. 

When $\cK_u$ is the covariance operator of a tempered $\gsp$ 
$u \in \cL ( \cS (\rr d) , L^2(\Omega) )$ then we denote the 
Weyl symbol for $\cK_u$ as $\sigma_u \in \cS'(\rr {2d})$. 
Hence we have the bijective correspondence
\begin{equation}\label{eq:KernelWeylsymbol}
k_u = (2 \pi)^{-\frac{d}{2}} ( \cF_2^{-1} \sigma_u ) \circ \kappa^{-1}
\quad \Longleftrightarrow \quad 
\sigma_u = (2 \pi)^{\frac{d}{2}} \cF_2 \left( k_u  \circ \kappa \right)
\end{equation}
between kernels and Weyl symbols of covariance operators. 
The non-negativity of the covariance operator, that is $( \cK_u \fy, \fy) \geqs 0$ for all $\fy \in \cS(\rr d)$, 
implies that the Weyl symbol is real-valued (see e.g. \cite[p.~30]{Cappiello1}): 
\begin{equation}\label{eq:RealWeylSymbol}
{\overline \sigma_u} = \sigma_u.
\end{equation}

From \cite[Proposition~4.3.2]{Grochenig1} it follows that $W(\wh g, \wh f) = W(g,f) \circ (-\J)$ for $f,g \in \cS(\rr d)$ where 
\begin{equation*}
\J =
\left(
\begin{array}{cc}
0 & I_d \\
-I_d & 0
\end{array}
\right) \in \rr {2d \times 2d}
\end{equation*}
is the matrix which plays a fundamental role in symplectic linear algebra \cite{Folland1,Grochenig1}.  
From Lemma \ref{lem:FourierCovOp} it follows that for $f,g \in \cS(\rr d)$
\begin{align*}
(2 \pi)^{-\frac{d}{2}} \left( \sigma_u, W(g,f) \right)
& = \left( \cK_u f, g \right)
= \left( \cK_{\wh u} \wh f, \wh g \right) \\
& = (2 \pi)^{-\frac{d}{2}} \left( \sigma_{\wh u}, W(\wh g, \wh f) \right) 
= (2 \pi)^{-\frac{d}{2}} \left( \sigma_{\wh u} \circ \J, W(g,f) \right)
\end{align*}
since $\J^{-1} = -\J$, and thus
\begin{equation}\label{eq:WeylSymbolFourierGSP}
\sigma_{\wh u} = \sigma_u \circ (- \J ). 
\end{equation}

If $\cK_u$ is the covariance operator of a stationary $\gsp$ $u$
then by
\eqref{eq:kernelstat}, 
\eqref{eq:SpectralMeasure} 
and \eqref{eq:KernelWeylsymbol}
its Weyl symbol is
\begin{equation}\label{eq:WeylSymbolStat}
\sigma_u = (2 \pi)^{\frac{d}{2}} 1 \otimes \wh h_u
= 1 \otimes \mu_u.
\end{equation}
If $u \in \cL ( \cS(\rr d) , L^2(\Omega) )$ is a frequency stationary $\gsp$ (cf. Definition \ref{def:FrequencyStat})
then by Lemma \ref{lem:FourierCovOp}, 
$\cK_{\wh u} = h_{\wh u} *$, 
\eqref{eq:convolutionFourier}, 
\eqref{eq:SpectralMeasure} 
and $\check f(x) = f(-x)$,
\begin{align*}
\cK_u f 
& = \cF^{-1} \left( h_{\wh u} * \wh f \right)
= \check{ \mu }_{\wh u} f
\end{align*}
which gives the corresponding Weyl symbol
\begin{equation}\label{eq:WeylSymbolFrequencyStat}
\sigma_u 
= \check{ \mu }_{\wh u} \otimes 1.
\end{equation}

The following statement will be needed in Section \ref{sec:wignergeneral}. 

\begin{lem}\label{lem:DensityWigner}
The space of finite linear combinations of functions of the form $W(f, g ) \in \cS(\rr {2d})$ with $f,g \in \cS(\rr d)$
is dense in $\cS(\rr {2d})$.
\end{lem}

\begin{proof}
The partial Fourier transform $\cF_2$ and composition with the invertible matrix $\kappa \in \rr {2d \times 2d}$ defined by \eqref{eq:kappa}
are isomorphisms on $\cS(\rr {2d})$. 
The claim thus reduces to the density in $\cS(\rr {2d})$ of finite linear combinations of functions of the form $f \otimes \overline g \in \cS(\rr {2d})$
with $f,g \in \cS(\rr d)$. 
The latter claim is a consequence of \cite[Theorem~V.13]{Reed1}.
\end{proof}

\subsection{White noise generalized stochastic processes}\label{subsec:WhiteNoise}

White noise is a fundamental concept in probability and engineering, with several different meanings \cite{Bardenet1,Biagini1,Feichtinger2,Flandrin3,Gelfand4,Hormann1,Keville1,Kuo1}. 
We will use the following definition \cite{Feichtinger2,Hormann1,Keville1}. 

\begin{defn}\label{def:WhiteNoise}
If $u \in \cL ( \cS(\rr d) , L^2(\Omega) )$ is a stationary $\gsp$ such that 
$(k_u, \fy \otimes \overline \psi) = p (\psi,\fy)_{L^2}$ for $\psi, \fy \in \cS(\rr d)$
with $p > 0$
then $u$ is called a white noise $\gsp$ with power $p$. 
\end{defn}

Definition \ref{def:WhiteNoise} can equivalently be formulated as 
$\cK_u = p I$, that is the covariance operator equals identity times $p$, or 
as the requirement that its Weyl symbol equals the power constantly: $\sigma_u = p$, 
or as $k_u = p (1 \otimes \delta_0 ) \circ \kappa^{-1}$ in terms of the Schwartz kernel.  
Finally we may say that the spectral measure $\mu_u = p (2 \pi)^{\frac{d}2} \wh{ \delta_0} = p$ is constantly equal to the power. 

Lemma \ref{lem:FourierCovOp} implies that $\cK_{\wh u} = p I$ if $u$ is white noise with power $p > 0$. 
Thus $u$ and $\wh u$ are simultaneously white noise with power $p > 0$.

\section{Gaussian symmetric tempered generalized stochastic processes}\label{sec:Gaussiantempered}

Let $u \in \cL ( \cS(\rr d) , L^2(\Omega) )$ be a tempered GSP. 
It is called Gaussian \cite{Gelfand4}
if for any finite set 
$\Phi=\{ \fy_j \}_{j=1}^{n} \subseteq \cS (\rr d)$, the $2n$-dimensional vector consisting 
of the real and the imaginary parts of the vector
$U_\Phi = \left( (u, \fy_1), \dots, (u, \fy_n) \right) \in \cc n$, that is 
$\wt U_\Phi = \left( \re \, U_\Phi, \ \im \, U_\Phi  \right) \in \rr {2n}$, 
$\wt U_\Phi: \Omega \to \rr {2n}$, 
has a $2n$-dimensional real Gaussian probability density. 
Define the mean $m_\Phi = \bE \wt U_\Phi$. 
In the case when the covariance matrix  
\begin{equation}
C_\Phi = \bE \left( \left( \wt U_\Phi - m_\Phi \right) \left( \wt U_\Phi - m_\Phi \right) ^T \right) \in \rr {2 n \times 2 n}
\end{equation}
is non-singular we then have for 
any $A \in \cB(\rr {2n})$ (the Borel $\sigma$-algebra),
that the probability for 
$\wt U_\Phi \in A$ is given by 
\begin{equation*}
(2\pi)^{-d} (\det C_\Phi)^{-1/2} \int_{A} 
e^{-\frac{1}{2} \la x - m_\Phi,  C_\Phi^{-1} (x - m_\Phi) \ra } \dd x. 
\end{equation*}
More generally, Gaussianity can be defined also when the covariance matrix is singular
using characteristic functions \cite{Doob1},
by requiring that for any finite set $\Phi=\{ \fy_j \}_{j=1}^{n} \subseteq \cS (\rr d)$
and any $y \in \rr {2n}$
we have 
\begin{equation*}
\bE \left( e^{i \la y, \wt U_\Phi \ra } \right)
= e^{ i \la y, m_\Phi \ra - \frac{1}{2} \la y, C_\Phi y \ra}.
\end{equation*}

We assume that complex-valued GSPs $u$ are \textit{symmetric} which means that $\lambda u$ and $u$ have identical probability
distribution laws for each $\lambda \in \co$ such that $|\lambda|=1$ 
\cite{Janson1,Miller1}. This is equivalent to 
$\bE (u,\fy)  = 0$ for all $\fy \in \cS(\rr d)$
and 
\begin{equation}\label{eq:covsymmetric}
\bE (  (u,\fy) (u,\psi) ) = 0 \quad \forall \fy,\psi \in \cS(\rr d),
\end{equation}
cf. \cite{Doob1,Janson1,Miller1}. 
Thus a symmetric GSP is automatically zero mean, that is, it belongs to GSP$_0$.

Let $x_1, \dots, x_n$ be complex-valued zero mean jointly Gaussian 
random variables. Isserlis' or Wick's theorem \cite{Janson1}
says that
\begin{equation*}
\bE(x_1 \cdots x_n) = \sum \prod_k \bE(x_{i_k} x_{j_k})
\end{equation*}
where the sum is taken over all partitions of $\{1, \dots ,n \}$ into 
disjoint pairs $\{i_k, j_k \}$. Thus for $n=4$ we have 
\begin{equation}\label{eq:wick}
\bE (x_1 x_2 x_3 x_4) 
= \bE(x_1 x_2) \bE(x_3 x_4) + \bE(x_1 x_3) \bE(x_2 x_4) + \bE(x_1 x_4) \bE(x_2 x_3). 
\end{equation}

Let  $u \in \cL ( \cS(\rr d) , L_0^2(\Omega) )$ be Gaussian and symmetric.
We define the sesquilinear form $U = u \otimes \overline u$ 
as
\begin{equation}\label{eq:Udef}
(U, \fy \otimes \overline \psi ) 
= ( u \otimes \overline u, \fy \otimes \overline \psi )
= (u,\fy) \overline{ (u,\psi) }, \quad \fy,\psi \in \cS(\rr d). 
\end{equation}
From \eqref{eq:wick} it follows that $\bE | (U, \fy \otimes \overline \psi ) |^2 = \bE \left( | (u,\fy) |^2 | (u,\psi) |^2\right)$ is finite which means that the form $U$ has values in $L^2(\Omega)$. 
Thus $U$ extends to
\begin{equation}\label{eq:Ugsp}
U \in \cL ( \cS(\rr {2d}) , L^2(\Omega) ).
\end{equation}
From \eqref{eq:crosscovdist} we identify
\begin{equation}\label{eq:EU}
\bE (U, \fy \otimes \overline \psi ) = ( k_ u,\fy \otimes  \overline{\psi }), \quad \fy,\psi \in \cS(\rr d),  
\end{equation}
so if we define 
\begin{equation}\label{eq:U0}
U_0 = U - \bE U = U - k_u
\end{equation}
then $\bE U_0 = 0$, that is, the sesquilinear form $U_0$ has values in $L_0^2(\Omega)$. 
By \eqref{eq:Ugsp}, \eqref{eq:EU} and the Schwartz kernel theorem \cite{Gelfand4}, \cite[Theorem~5.2.1]{Hormander1}
the form $U_0$ extends to a tempered GSP$_0$, still denoted $U_0 \in \cL ( \cS(\rr {2d}) , L_0^2(\Omega) )$. 
Again by the Schwartz kernel theorem it follows that its covariance distribution 
satisfies $k_{U_0} \in  \cS' (\rr {4d})$. 

Let $\fy_j,\psi_j \in \cS(\rr d)$ for $j = 1,2$.
From \eqref{eq:Udef} and \eqref{eq:U0} we get
\begin{align*}
& ( U_0, \fy_1 \otimes \overline \psi_1 ) \overline{ (U_0, \fy_2 \otimes \overline \psi_2 ) } \\
& = 
\left( (u,\fy_1) \overline{ (u,\psi_1) }  - ( k_u, \fy_1 \otimes \overline \psi_1 ) \right)
\overline{ \left( (u,\fy_2) \overline{ (u,\psi_2) }  - ( k_u, \fy_2 \otimes \overline \psi_2 ) \right) }
\end{align*}
which gives using 
\eqref{eq:covsymmetric} and \eqref{eq:wick}
\begin{align*}
& \bE \left( 
(U_0, \fy_1 \otimes \overline \psi_1 ) \overline{ (U_0, \fy_2 \otimes \overline \psi_2 ) }
\right) \\
& \bE \left( 
(u,\fy_1) \overline{ (u,\psi_1)  (u,\fy_2)}  (u,\psi_2)
\right) 
- (k_u, \fy_1 \otimes \overline \psi_1)  
\overline{ (k_u, \fy_2 \otimes \overline \psi_2) } \\ 
& = 
\bE \left( (u,\fy_1) \overline{ (u,\psi_1) }  \right) 
\bE \left( \overline{ (u,\fy_2) } (u,\psi_2)  \right) 
+ 
\bE \left( (u,\fy_1) \overline{ (u,\fy_2) }  \right) 
\bE \left( \overline{ (u,\psi_1) } (u,\psi_2)  \right) \\
& - (k_u, \fy_1 \otimes \overline \psi_1)  
\overline{ (k_u, \fy_2 \otimes \overline \psi_2) } \\
& = 
(k_u, \fy_1 \otimes \overline \fy_2)  
\overline{ (k_u, \psi_1 \otimes \overline \psi_2) }.
\end{align*}

Thus
\begin{align*}
( k_{U_0}, \fy_1 \otimes \overline \psi_1 \otimes \overline \fy_2 \otimes \psi_2 )
& = ( k_u \otimes \overline{k_u}, \fy_1 \otimes \overline \fy_2 \otimes \overline \psi_1 \otimes \psi_2 ). 
\end{align*}
By the density of simple tensors in $\cS(\rr {4d})$ (cf. \cite[Theorem~V.13]{Reed1}) it follows that 
\begin{equation}\label{eq:kU0}
k_{U_0} = \left( k_u \otimes \overline{k_u} \right) \circ p_{2,3}
\end{equation}
where $p_{2,3}: \rr {4d} \to \rr {4d}$ is the linear coordinate transformation
\begin{equation}\label{eq:p23}
p_{2,3}(x_1, y_1, x_2, y_2) = (x_1, x_2, y_1, y_2), \quad x_1, x_2, y_1, y_2 \in \rr d, 
\end{equation}
that is, the transposition of the second and the third $d$ coordinates in $\rr {4d}$.

\section{The Wigner distribution of tempered symmetric Gaussian $\gsp$s}\label{sec:wignergeneral}

The Wigner distribution of a deterministic Schwartz function $f$ is defined by \eqref{eq:WignerSchwartz} with $g = f$. 
By \eqref{eq:kappainv} and the principles of distribution theory \cite{Hormander1} the Wigner distribution $W(u) \in \cS'(\rr {2d})$ of $u \in \cS'(\rr d)$
can be defined naturally as 
\begin{equation}\label{eq:Wignertempdist}
\left( W(u), \Phi \right)
= \left( u \otimes \overline u, \left( \cF_2^{-1} \Phi \right) \circ \kappa^{-1} \right), \quad \Phi \in \cS(\rr {2d}). 
\end{equation}

Now we want to define the Wigner distribution of a tempered, Gaussian and symmetric $\gsp$
$u \in \cL ( \cS(\rr d) , L_0^2(\Omega) )$.
Following \eqref{eq:Wignertempdist} we define the Wigner distribution $W(u)$ of $u$ as 
\begin{equation}\label{eq:Wdefgsp}
\left( W(u), \Phi \right)
= \left( U, \left( \cF_2^{-1} \Phi \right) \circ \kappa^{-1} \right), \quad \Phi \in \cS(\rr {2d}), 
\end{equation}
with $U = u \otimes \overline u \in \cL ( \cS(\rr {2d} ) , L^2(\Omega) )$ defined by \eqref{eq:Udef}.
It follows that $W(u) \in \cL ( \cS(\rr {2d} ) , L^2(\Omega) )$. 
Moreover \eqref{eq:EU} and \eqref{eq:KernelWeylsymbol} give for $\Phi \in \cS(\rr {2d})$
\begin{equation}\label{eq:MeanWigner}
\begin{aligned}
\bE \left( W(u), \Phi \right)
& = \bE  \left( U, \left( \cF_2^{-1} \Phi \right) \circ \kappa^{-1} \right) 
= \left( k_u, \left( \cF_2^{-1} \Phi \right) \circ \kappa^{-1} \right) \\
& = \left( \cF_2 \left( k_u \circ \kappa \right), \Phi \right) 
= (2 \pi)^{-\frac{d}{2}} \left( \sigma_u, \Phi \right) 
\end{aligned}
\end{equation}
and thus 
\begin{equation}\label{eq:EW}
\bE W(u) = (2 \pi)^{-\frac{d}{2}} \sigma_u. 
\end{equation}
The expected value of the Wigner distribution $W(u)$ hence equals 
the Weyl symbol of the covariance operator of $u$, times $(2 \pi)^{-\frac{d}{2}}$ \cite{Hormann1,Janssen1}.

If $u \in \cL ( \cS(\rr d) , L_0^2(\Omega) )$ is Gaussian and symmetric
then we define its zero mean Wigner distribution using  \eqref{eq:Udef}, 
\eqref{eq:EU} and \eqref{eq:U0} as
\begin{equation}\label{eq:W0}
( W_0(u), \Phi ) = \left( U_0, \left( \cF_2^{-1} \Phi \right) \circ \kappa^{-1} \right), \quad \Phi \in \cS(\rr {2d}), 
\end{equation}
which means that $W_0(u) \in \cL ( \cS(\rr {2d} ) , L_0^2(\Omega) )$ is a tempered $\gsp_0$ defined on $\rr {2d}$.
By \eqref{eq:MeanWigner} we have
\begin{equation*}
\left( \bE U, \left( \cF_2^{-1} \Phi \right) \circ \kappa^{-1} \right)
= \left( k_u, \left( \cF_2^{-1} \Phi \right) \circ \kappa^{-1} \right)
= \left( \bE W(u), \Phi \right)
\end{equation*}
and it follows from \eqref{eq:EW} that  
\begin{equation}\label{eq:WignerZMWigner}
W_0(u) = W(u) - (2 \pi)^{-\frac{d}{2}} \sigma_u. 
\end{equation}

Let $\sigma_W \in \cS' (\rr {4d})$ denote the Weyl symbol of the covariance operator of $W_0(u)$, and $k_{W_0(u)}  \in \cS' (\rr {4d})$ its Schwartz kernel. 
By \eqref{eq:KernelWeylsymbol} we have 
\begin{equation}\label{eq:KernelWeylsymbolWigner}
k_{W_0(u)} = (2 \pi)^{-d} ( \cF_2^{-1} \sigma_W ) \circ \kappa_2^{-1}
\end{equation}
where 
\begin{equation}\label{eq:kappa2}
\kappa_2 = 
\left(
\begin{array}{ll}
I_{2 d} & \frac12 I_{2 d} \\
I_{2 d} & - \frac12 I_{2 d}
\end{array}
\right) \in \rr {4 d \times 4 d}
\end{equation}
and
\begin{equation}\label{eq:kappa2inv}
\kappa_2^{-1} = 
\left(
\begin{array}{ll}
\frac12 I_{2 d} & \frac12 I_{2 d} \\
I_{2 d} & - I_{2 d}
\end{array}
\right) \in \rr {4 d \times 4 d}.
\end{equation}

The next lemma relates the Wigner distributions of $u$ and $\wh u$. 

\begin{lem}\label{lem:WignerFourier}
If $u \in \cL ( \cS(\rr d) , L_0^2(\Omega) )$ is Gaussian and symmetric then 
\begin{align}
W( \wh u) & = W( u) \circ (-\J) \quad \mbox{and} \quad \label{eq:WignerFourierGSP} \\
W_0( \wh u) & = W_0( u) \circ (-\J). \label{eq:WignerFourierGSPZM}
\end{align}
\end{lem}

\begin{proof}
Let $f,g \in \cS(\rr d)$. 
By \cite[Proposition~4.3.2]{Grochenig1} we have $W( \wh f, \wh g) = W(f,g) \circ (-\J)$.
Combining this with \eqref{eq:Wignertempdist} and \eqref{eq:WignerSchwartz} we get 
\begin{align*}
\left( W(\wh u) \circ \J, W(f,g) \right)
& = \left( W(\wh u), W(\wh f, \wh g) \right)
= \left( \wh u \otimes \overline{\wh u}, \left( \cF_2^{-1} W(\wh f, \wh g) \right) \circ \kappa^{-1} \right) \\
& = \left( \wh u \otimes \overline{\wh u}, \wh f \otimes \overline{\wh g} \right)
= \left( u \otimes \overline{u}, f \otimes \overline{g} \right)
= \left( W(u), W(f,g) \right).
\end{align*}
By Lemma \ref{lem:DensityWigner} this proves \eqref{eq:WignerFourierGSP}. 
Finally \eqref{eq:WignerFourierGSPZM} is a consequence of \eqref{eq:WignerFourierGSP}, 
\eqref{eq:WignerZMWigner}
and \eqref{eq:WeylSymbolFourierGSP}. 
\end{proof}

Referring to \eqref{eq:kappa} and \eqref{eq:kappainv} we next introduce the block matrix notations
\begin{equation*}
\kappa \otimes \kappa = 
\left(
\begin{array}{llll}
I_d & \frac12 I_d & 0 & 0 \\
I_d & - \frac12 I_d & 0 & 0 \\
0 & 0 & I_d & \frac12 I_d \\
0 & 0 & I_d & -\frac12 I_d 
\end{array}
\right) \in \rr {4d \times 4d}, 
\end{equation*}
\begin{equation*}
\kappa^{-1} \otimes \kappa^{-1} = 
\left(
\begin{array}{llll}
\frac12 I_d & \frac12 I_d & 0 & 0 \\
I_d & -  I_d & 0 & 0 \\
0 & 0 & \frac12 I_d & \frac12 I_d \\
0 & 0 & I_d & - I_d 
\end{array}
\right) \in \rr {4d \times 4d}, 
\end{equation*}
and 
\begin{equation*}
I \otimes \J = 
\left(
\begin{array}{ll}
I_{2d} & 0 \\
0 & \J 
\end{array}
\right) \in \rr {4d \times 4d}.
\end{equation*}

With these ingredients and \eqref{eq:p23} we can formulate the following formula which will be 
crucial for our main theorem. 

\begin{lem}\label{lem:Fouriertensor}
If $\Phi, \Psi \in \cS(\rr {2d})$ then 
\begin{align*}
\kF_2 \kF_4^{-1} \Big( \left( \cF_2^{-1} \Phi \otimes \cF_2 \overline{ \Psi } \right) \circ (\kappa^{-1} \otimes \kappa^{-1}) \circ p_{2,3} \circ ( \kappa \otimes \kappa ) \Big) 
= W(\Phi, \Psi ) \circ \left( I \otimes \J \right) \circ \kappa_2^{-1}.
\end{align*}
\end{lem}

\begin{proof}
First we note that if $f,g \in \cS(\rr {2d})$ in \eqref{eq:WignerSchwartz}  then $\cF_2$ denotes the partial Fourier
transform with respect to the second $2d$ coordinate in $\rr {4d}$. It can be written $\kF_3 \kF_4$
where $\kF_j$ denotes the partial Fourier transform with respect to the $j$th $d$ variables in $\rr {4d}$.

Hence from \eqref{eq:WignerSchwartz} we get
\begin{equation*}
 \cF_2^{-1} \Phi \otimes \cF_2 \overline{ \Psi }
 = \kF_2^{-1} \kF_4 \left( \Phi \otimes \overline{ \Psi } \right) 
 = \kF_2^{-1} \kF_4 \left( \left(  \kF_3^{-1}  \kF_4^{-1}  W \left(\Phi, \Psi \right) \right) \circ \kappa_2^{-1} \right) 
\end{equation*}
which gives for $x_1, \xi_1, x_2, \xi_2 \in \rr d$
\begin{equation}\label{eq:tensorproduct1}
\begin{aligned}
& \left( \cF_2^{-1} \Phi \otimes \cF_2 \overline{ \Psi } \right) (x_1, \xi_1, x_2, \xi_2) \\
& = (2 \pi)^{-d} \iint_{\rr {2d}} 
\kF_3^{-1}  \kF_4^{-1}  W(\Phi, \Psi ) 
\left( \frac12 \left( x_1 + x_2 \right), \frac12 \left( y_1 + y_2 \right), x_1 - x_2, y_1 - y_2 \right)
e^{i \left( \la y_1, \xi_1 \ra - \la y_2, \xi_2 \ra \right)}
\dd y_1 \dd y_2 \\
& = (2 \pi)^{-d} \iint_{\rr {2d}} 
\kF_3^{-1}  \kF_4^{-1}  W(\Phi, \Psi ) 
\left( \frac12 \left( x_1 + x_2 \right), u_1, x_1 - x_2, u_2 \right)
e^{i \left( \la u_1 + \frac12 u_2 , \xi_1 \ra - \la u_1 - \frac12 u_2, \xi_2 \ra \right)}
\dd u_1 \dd u_2 \\
& = \kF_2^{-1} \kF_3^{-1}  W(\Phi, \Psi ) 
\left( \frac12 \left( x_1 + x_2 \right), \xi_1 - \xi_2, x_1 - x_2, - \frac12 \left( \xi_1 + \xi_2 \right) \right). 
\end{aligned}
\end{equation}

We have 
\begin{equation*}
( \kappa^{-1} \otimes \kappa^{-1}) \circ p_{2,3} \circ ( \kappa \otimes \kappa )
= \left(
\begin{array}{llll}
\frac12 I_d & \frac14 I_d & \frac12 I_d & \frac14 I_d \\ 
I_d & \frac12 I_d & - I_d  & - \frac12 I_d \\
\frac12 I_d & - \frac14 I_d & \frac12 I_d & -\frac14 I_d \\
I_d & -\frac12 I_d & - I_d & \frac12 I_d 
\end{array}
\right) \in \rr {4d \times 4d}
\end{equation*}
which combined with \eqref{eq:tensorproduct1} finally yields
\begin{align*}
& \kF_2 \kF_4^{-1} \Big( \left( \cF_2^{-1} \Phi \otimes \cF_2 \overline{ \Psi } \right) \circ ( \kappa^{-1} \otimes \kappa^{-1}) \circ p_{2,3} \circ ( \kappa \otimes \kappa ) \Big) 
(x_1, \xi_1, x_2, \xi_2) \\
& = (2 \pi)^{-d} \iint_{\rr {2d}}
\cF_2^{-1} \Phi \left( \frac12 \left( x_1 + x_2 \right) + \frac14 \left( y_1 + y_2 \right), x_1 - x_2 +  \frac12 \left( y_1 - y_2 \right) \right)  \\
& \qquad \qquad \times \cF_2 \overline{ \Psi } \left( \frac12 \left( x_1 + x_2 \right) - \frac14 \left( y_1 + y_2 \right), x_1 - x_2  -  \frac12 \left( y_1 - y_2 \right) \right)
e^{i \left(  \la y_2, \xi_2 \ra  - \la y_1, \xi_1 \ra \right)} \dd y_1 \dd y_2 \\
& = (2 \pi)^{-d} \iint_{\rr {2d}}
\kF_2^{-1} \kF_3^{-1}  W(\Phi, \Psi ) 
\left( \frac12 \left( x_1 + x_2 \right), y_1 - y_2, \frac12 \left( y_1 + y_2 \right), - (x_1 - x_2) \right) \\
& \qquad \qquad \qquad \qquad \qquad \qquad \qquad \qquad \qquad \qquad \qquad \qquad  
\times e^{i \left(  \la y_2, \xi_2 \ra  - \la y_1, \xi_1 \ra \right)} \dd y_1 \dd y_2 \\
& = (2 \pi)^{-d} \iint_{\rr {2d}}
\kF_2^{-1} \kF_3^{-1}  W(\Phi, \Psi ) 
\left( \frac12 \left( x_1 + x_2 \right), u_2, u_1, - (x_1 - x_2) \right) \\
& \qquad \qquad \qquad \qquad \qquad \qquad \qquad \qquad \qquad \qquad   
\times e^{i \left(  \la u_1 - \frac12 u_2, \xi_2 \ra  - \la u_1 + \frac12 u_2, \xi_1 \ra \right)} \dd u_1 \dd u_2 \\
& = W(\Phi, \Psi ) 
\left( \frac12 \left( x_1 + x_2 \right), \frac12 \left( \xi_1 + \xi_2 \right), \xi_1 - \xi_2, - (x_1 - x_2) \right) \\
& = W(\Phi, \Psi ) 
\left( \left( I \otimes \J \right) \left( \frac12 \left( x_1 + x_2 \right), \frac12 \left( \xi_1 + \xi_2 \right), x_1 - x_2, \xi_1 - \xi_2 \right) \right) \\
& = W(\Phi, \Psi ) 
\left( \left( I \otimes \J \right) \circ \kappa_2^{-1} (x_1, \xi_1, x_2, \xi_2) \right). 
\end{align*}
\end{proof}

We have prepared the ground for the statement and short proof of our following main result. 

\begin{thm}\label{thm:symbol}
Suppose $u \in \cL ( \cS(\rr d ) , L_0^2(\Omega) )$ is a symmetric Gaussian tempered $\gsp$
and let $\sigma_u \in \cS' (\rr {2d})$ denote the Weyl symbol of its covariance operator. 
Define the zero mean Wigner distribution $W_0(u) \in \cL ( \cS(\rr {2d} ) , L_0^2(\Omega) )$ 
by \eqref{eq:Udef}, \eqref{eq:U0} and \eqref{eq:W0}.  
If $\sigma_W \in \cS' (\rr {4d})$ denotes the Weyl symbol of the covariance operator of $W_0(u)$ 
then 
\begin{equation*}
\sigma_W = 
\left( \sigma_u \otimes \sigma_u \right) \circ \kappa_2 \circ \left( I \otimes (-\J) \right).
\end{equation*}
\end{thm}

\begin{proof}
A combination of 
\eqref{eq:WignerSchwartz},
\eqref{eq:KernelWeylsymbol}, 
\eqref{eq:RealWeylSymbol}, 
\eqref{eq:kU0}, 
\eqref{eq:W0}, 
\eqref{eq:KernelWeylsymbolWigner}, 
and Lemma \ref{lem:Fouriertensor}
gives for $\Phi, \Psi \in \cS(\rr {2d})$
\begin{align*}
& (2 \pi)^{-d} \left( \sigma_W, W(\Phi, \Psi) \right) \\
& = \left( k_{W_0(u)}, \Phi \otimes \overline \Psi \right) \\
& = \bE \left( \left( U_0, \left( \cF_2^{-1} \Phi \right) \circ \kappa^{-1} \right) \overline{ \left( U_0, \left( \cF_2^{-1} \Psi \right) \circ \kappa^{-1} \right) } \right) \\
& = \left( k_{U_0}, \left( \cF_2^{-1} \Phi \right) \circ \kappa^{-1} \otimes \overline{ \left( \cF_2^{-1} \Psi \right) \circ \kappa^{-1} } \right) \\
& = \left( k_u \otimes \overline k_u, \left( \cF_2^{-1} \Phi \otimes \cF_2 \overline{ \Psi } \right) \circ ( \kappa^{-1} \otimes \kappa^{-1}) \circ p_{2,3} \right) \\
& = (2 \pi)^{-d} \left( (\cF_2^{-1} \sigma_u ) \circ \kappa^{-1} \otimes (\cF_2  \sigma_u ) \circ \kappa^{-1}, \left( \cF_2^{-1} \Phi \otimes \cF_2 \overline{ \Psi } \right) \circ ( \kappa^{-1} \otimes \kappa^{-1}) \circ p_{2,3} \right) \\
& = (2 \pi)^{-d} \left( \sigma_u \otimes \sigma_u, 
\kF_2 \kF_4^{-1} \left( \left( \cF_2^{-1} \Phi \otimes \cF_2 \overline{ \Psi } \right) \circ ( \kappa^{-1} \otimes \kappa^{-1}) \circ p_{2,3} \circ ( \kappa \otimes \kappa ) \right) \right) \\
& = (2 \pi)^{-d} \left( \sigma_u \otimes \sigma_u, W(\Phi, \Psi ) \circ \left( I \otimes \J \right) \circ \kappa_2^{-1} \right) \\
& = (2 \pi)^{-d} \left( \left( \sigma_u \otimes \sigma_u \right) \circ \kappa_2 \circ \left( I \otimes (-\J) \right), W(\Phi, \Psi ) \right). 
\end{align*}
The claim now follows from Lemma \ref{lem:DensityWigner}. 
\end{proof}

If we write out the coordinates then Theorem \ref{thm:symbol} reads for $x_1, x_2, \xi_1, \xi_2 \in \rr d$
\begin{equation}\label{eq:WeylWignerCoord}
\sigma_W (x_1, x_2, \xi_1, \xi_2) = 
\sigma_u \left( x_1 - \frac12 \xi_2, x_2 + \frac12 \xi_1 \right) \sigma_u \left( x_1 + \frac12 \xi_2, x_2 - \frac12 \xi_1 \right). 
\end{equation}
%

\section{The Wigner distribution for symmetric Gaussian stationary $\gsp$s and white noise}\label{sec:wignerstationary}

In this section we draw some conclusions from Theorem \ref{thm:symbol} for the particular cases of 
stationary, frequency stationary, and white noise $\gsp$s. 

\begin{prop}\label{prop:WignerStat}
Suppose $u \in \cL ( \cS(\rr d ) , L_0^2(\Omega) )$ is a symmetric, Gaussian and stationary $\gsp$
with non-negative spectral Radon measure $\mu_u$. 
Define the Wigner distribution $W(u) \in \cL ( \cS(\rr {2d} ) , L^2(\Omega) )$ 
by \eqref{eq:Udef} and \eqref{eq:Wdefgsp}, 
and define the zero mean Wigner distribution $W_0(u) \in \cL ( \cS(\rr {2d} ) , L_0^2(\Omega) )$ 
by \eqref{eq:Udef}, \eqref{eq:U0} and \eqref{eq:W0}.  
Denote by $\sigma_W \in \cS' (\rr {4d})$ the Weyl symbol of the covariance operator of $W_0(u)$.
Then
\begin{align}
\bE W(u) & = (2 \pi)^{-\frac{d}{2}} 1 \otimes \mu_u, \quad \mbox{and} \quad  \label{eq:WignerStatExpect} \\
\sigma_W (x_1, x_2, \xi_1, \xi_2)
& = \mu_u \left(  x_2 + \frac12 \xi_1 \right) \mu_u \left(  x_2 - \frac12 \xi_1 \right), 
\quad x_1, x_2, \xi_1, \xi_2 \in \rr d. \label{eq:WignerStatWeyl}
\end{align}
\end{prop}

\begin{proof}
Formula \eqref{eq:WignerStatExpect} is a direct consequence of \eqref{eq:EW} and \eqref{eq:WeylSymbolStat}, 
and \eqref{eq:WignerStatWeyl} follows from \eqref{eq:WeylSymbolStat} and Theorem \ref{thm:symbol} (cf. \eqref{eq:WeylWignerCoord}). 
\end{proof}

\begin{cor}\label{cor:StatTranslationSymbol}
Under the assumptions of Proposition \ref{prop:WignerStat} we have 
\begin{equation*}
T_{x, 0, 0, \xi} \sigma_W  = \sigma_W \quad \forall x, \xi \in \rr d, 
\end{equation*}
that is $\sigma_W$ is invariant under translation in the first and fourth $\rr d$ coordinate. 
\end{cor}

\begin{cor}\label{cor:StatTranslationInvariance}
Under the assumptions of Proposition \ref{prop:WignerStat} we have 
\begin{equation*}
\cK_{W_0(u)} T_{(x,0)} M_{(0, \xi)} 
= T_{(x,0)} M_{(0, \xi)} \cK_{W_0(u)}
 \quad \forall x, \xi \in \rr d.  
\end{equation*}
\end{cor}

\begin{proof}
From Corollary \ref{cor:StatTranslationSymbol} and 
\cite[Proposition~4.3.2]{Grochenig1}
we obtain for $\Phi, \Psi \in \cS(\rr {2d})$ and any $x, \xi \in \rr d$
\begin{align*}
\left( \cK_{W_0(u)} \Psi, \Phi \right)
& = (2 \pi)^{-d} \left( \sigma_W, W(\Phi, \Psi) \right) \\
& = (2 \pi)^{-d} \left( \sigma_W, T_{x, 0, 0, \xi} W(\Phi, \Psi) \right) \\
& = (2 \pi)^{-d} \left( \sigma_W, W( T_{(x,0)} M_{(0, \xi)} \Phi, T_{(x,0)} M_{(0, \xi)} \Psi) \right) \\
& = \left( \cK_{W_0(u)} T_{(x,0)} M_{(0, \xi)}  \Psi, T_{(x,0)} M_{(0, \xi)}  \Phi \right) \\
& = \left( M_{(0, - \xi)}   T_{(-x,0)} \cK_{W_0(u)} T_{(x,0)} M_{(0, \xi)}  \Psi, \Phi \right)
\end{align*}
which implies $\cK_{W_0(u)} T_{(x,0)} M_{(0, \xi)} = T_{(x,0)} M_{(0, \xi)}  \cK_{W_0(u)}$
for all $x, \xi \in \rr d$. 
\end{proof}

\begin{rem}\label{rem:StatWigner}
Comparing Corollary \ref{cor:StatTranslationInvariance} with \eqref{eq:TranslationInvariance} and 
\eqref{eq:ModulationInvariance} we see that 
the zero mean Wigner distribution $W_0(u)$ for a stationary Gaussian symmetric $\gsp$
has stationary behavior in the first (``time'') $\rr d$ coordinate in $T^* \rr d$, and 
frequency stationary behavior in the second (``frequency'') $\rr d$ coordinate in $T^* \rr d$. 
These consequences of stationarity are quite natural, considering the zero mean Wigner distribution $W_0(u)$
as a time-frequency representation of $u$.
\end{rem}

\begin{cor}\label{cor:WignerWGN}
Suppose $u \in \cL ( \cS(\rr d ) , L_0^2(\Omega) )$ is a symmetric, Gaussian $\gsp$ which is white noise with power $p > 0$. 
Then the Wigner distribution $W(u) \in \cL ( \cS(\rr {2d} ) , L^2(\Omega) )$ 
has mean 
\begin{equation}\label{eq:WSwhitenoise}
\bE W(u) (x,\xi) = (2 \pi)^{- \frac{d}{2} } p, \quad (x,\xi) \in T^* \rr d, 
\end{equation}
and the zero mean Wigner distribution $W_0(u) \in \cL ( \cS(\rr {2d} ) , L_0^2(\Omega) )$ 
is white noise with power $p^2$. 
\end{cor}

\begin{proof}
According to the discussion after Definition \ref{def:WhiteNoise} we have $\mu_u = p$ and $\sigma_u = p$.
The claim \eqref{eq:WSwhitenoise} thus follows from 
\eqref{eq:WignerStatExpect} in Proposition \ref{prop:WignerStat}. 

From \eqref{eq:WignerStatWeyl} in Proposition \ref{prop:WignerStat} it likewise follows that $\sigma_W = p^2$ on $T^* \rr {2d}$. 
Hence the covariance operator for $W_0(u)$ is $\cK_{W_0(u)} = p^2 I$.
By Definition \ref{def:WhiteNoise} it follows that $W_0(u)$ is white noise on $T^* \rr d$ with power $p^2$. 
\end{proof}

Finally we state versions of Proposition \ref{prop:WignerStat}, 
and Corollaries \ref{cor:StatTranslationSymbol} and \ref{cor:StatTranslationInvariance}
for frequency stationary $\gsp$s 
according to Definition \ref{def:FrequencyStat}.

\begin{prop}\label{prop:WignerFrequencyStat}
Suppose $u \in \cL ( \cS(\rr d ) , L_0^2(\Omega) )$ is a symmetric, Gaussian and frequency stationary 
$\gsp$ such that $\wh u$ has non-negative spectral Radon measure $\mu_{\wh u}$. 
Define the Wigner distribution $W(u) \in \cL ( \cS(\rr {2d} ) , L^2(\Omega) )$ 
by \eqref{eq:Udef} and \eqref{eq:Wdefgsp}, 
and define the zero mean Wigner distribution $W_0(u) \in \cL ( \cS(\rr {2d} ) , L_0^2(\Omega) )$ 
by \eqref{eq:Udef}, \eqref{eq:U0} and \eqref{eq:W0}.  
Denote by $\sigma_W \in \cS' (\rr {4d})$ the Weyl symbol of the covariance operator of $W_0(u)$.
Then
\begin{align}
\bE W(u) & = (2 \pi)^{-\frac{d}{2}} \check{\mu}_{\wh u} \otimes 1,  \label{eq:WignerFStatExpect} \\
\sigma_W (x_1, x_2, \xi_1, \xi_2)
& = \mu_{\wh u} \left(  -x_1 + \frac12 \xi_2 \right) \mu_{\wh u} \left(  -x_1 - \frac12 \xi_2 \right), 
\quad x_1, x_2, \xi_1, \xi_2 \in \rr d, \label{eq:WignerFStatWeyl}
\end{align}
\begin{equation}\label{eq:TranslationInvarianceFstat}
T_{0,x,\xi,0} \sigma_W  = \sigma_W \quad \forall x, \xi \in \rr d, 
\end{equation}
and 
\begin{equation}\label{eq:TranslationCommFstat}
\cK_{W_0(u)} T_{(0,x)} M_{(\xi, 0)} 
= T_{(0, x)} M_{(\xi, 0)} \cK_{W_0(u)}
 \quad \forall x, \xi \in \rr d.  
\end{equation}
\end{prop}

\begin{proof}
The assumptions and \eqref{eq:WignerStatExpect} of Proposition \ref{prop:WignerStat} imply
$\bE W(\wh u) = (2 \pi)^{-\frac{d}{2}} 1 \otimes \mu_{\wh u}$.
Lemma \ref{lem:WignerFourier} thus gives 
\begin{align*}
\bE W(u) 
& = \bE W(\wh u) \circ \J 
= (2 \pi)^{-\frac{d}{2}} \check{\mu}_{\wh u} \otimes 1
\end{align*}
which is \eqref{eq:WignerFStatExpect}.

The formula \eqref{eq:WignerFStatWeyl} is a consequence of Theorem \ref{thm:symbol}, 
\eqref{eq:WeylSymbolFourierGSP} and $\sigma_{\wh u} = 1 \otimes \mu_{\wh u}$, cf. \eqref{eq:WeylSymbolStat}. 
The translation invariance \eqref{eq:TranslationInvarianceFstat} follows from \eqref{eq:WignerFStatWeyl}, 
and finally \eqref{eq:TranslationCommFstat} follows from \eqref{eq:TranslationInvarianceFstat} 
and \cite[Proposition~4.3.2]{Grochenig1} as
\begin{align*}
\left( \cK_{W_0(u)} \Psi, \Phi \right)
& = (2 \pi)^{-d} \left( \sigma_W, W(\Phi, \Psi) \right) \\
& = (2 \pi)^{-d} \left( \sigma_W, T_{0,x, \xi,0} W(\Phi, \Psi) \right) \\
& = (2 \pi)^{-d} \left( \sigma_W, W( T_{(0,x)} M_{(\xi,0)} \Phi, T_{(0,x)} M_{(\xi,0)} \Psi) \right) \\
& = \left( \cK_{W_0(u)} T_{(0,x)} M_{(\xi, 0)}  \Psi, T_{(0, x)} M_{(\xi, 0)}  \Phi \right) \\
& = \left( M_{(- \xi,0)}   T_{(0, -x)} \cK_{W_0(u)} T_{(0, x)} M_{(\xi, 0)}  \Psi, \Phi \right)
\end{align*}
for all $\Phi, \Psi \in \cS(\rr {2d})$ and all $x, \xi \in \rr d$.
\end{proof}

\begin{rem}\label{rem:FrequencyStatWigner}
Again comparing \eqref{eq:TranslationCommFstat} with \eqref{eq:TranslationInvariance} and 
\eqref{eq:ModulationInvariance} we see that 
the zero mean Wigner distribution $W_0(u)$ for a frequency stationary Gaussian symmetric $\gsp$
has stationary behavior in the second (``frequency'') $\rr d$ coordinate in $T^* \rr d$, and 
frequency stationary behavior in the first (``time'') $\rr d$ coordinate in $T^* \rr d$. 
\end{rem}

\section{The Wigner distribution of symmetric Brownian motion}\label{sec:WignerBrown}

In this section we apply Theorem \ref{thm:symbol} to Brownian motion. 
Consider a Wiener process (Brownian motion) $b: \ro_+ \to L_0^2(\Omega)$ which is real-valued Gaussian with covariance function
\begin{equation}\label{eq:CovBrown}
\bE \left( b(x) b(y) \right)
= \frac12 \min(x,y), \quad x, y \in \ro_+.
\end{equation}
A symmetric Gaussian complex-valued Wiener process can be constructed as $b = b_1 + i b_2$
where $b_1$ and $b_2$ are independent real-valued Gaussian with covariance function \eqref{eq:CovBrown}.
It follows that 
\begin{equation*}
\bE \left( b(x) \overline{b(y)} \right) = \min(x,y), \quad x, y \in \ro_+.
\end{equation*}

We may extend $b$ to domain $\ro$ by defining $b(x) = 0$ for $x < 0$, 
and then 
\begin{equation}\label{eq:covbrown}
k_b(x,y) = \bE \left( b(x) \overline{b(y)} \right) = 
\left\{
\begin{array}{ll}
\min(x,y), & x, y \in \ro_+, \\
0, & x < 0 \mbox{ or } y < 0.
\end{array}
\right.
\end{equation}
From $k_b \in \cS'(\rr 2)$ 
it follows that 
\begin{equation*}
(b, \fy) = \int_\ro b(x) \overline{\fy(x)} \dd x, \quad \fy \in \cS(\ro),
\end{equation*}
defines $b \in \cL( \cS(\ro), L_0^2(\Omega))$ that is a tempered non-stationary $\gsp$. 

Denoting the Heaviside step function by $h$ we have if $\fy, \psi \in \cS(\ro)$ 
\begin{align*}
\iint_{\rr 2} k_b(x,y) \overline{\fy'(x)} \psi'(y) \dd x \dd y 
& = \iint_{0 \leqs x \leqs y} x \overline{\fy'(x)} \psi'(y) \dd x \dd y +
\iint_{0 \leqs y \leqs x} y \overline{\fy'(x)} \psi'(y) \dd x \dd y \\
& = \int_{y = 0}^{+\infty} \psi'(y) \left( \int_{x = 0}^y x \overline{\fy'(x)} \dd x  \right) \dd y
+ \int_{y = 0}^{+\infty} \psi'(y) y \left( \int_{x = y}^{+\infty} \overline{\fy'(x)} \dd x  \right) \dd y \\
& = \int_{y = 0}^{+\infty} \psi'(y) \left( \left[ x \overline{\fy(x)} \right]_0^y - \int_{x = 0}^y \overline{\fy(x)} \dd x  \right) \dd y
- \int_{y = 0}^{+\infty} \psi'(y) y \overline{\fy(y)} \dd y \\
& = - \int_{y = 0}^{+\infty} \psi'(y) \left( \int_{x = 0}^y \overline{\fy(x)} \dd x  \right) \dd y \\
& = - \int_{x = 0}^{+\infty} \overline{\fy(x)} \left( \int_{y = x}^{+\infty} \psi'(y) \dd y  \right) \dd x \\
& = \int_{-\infty}^{+\infty} h (x) \psi(x) \overline{h (x) \fy(x)} \dd x \\
& = \left( 1 \otimes \delta_0, \left( h \fy \otimes \overline{h \psi} \right) \circ \kappa \right) \\
& = \left( \left( h \otimes h \right) \left( \left( 1 \otimes \delta_0 \right) \circ \kappa^{-1} \right), \fy \otimes \overline{\psi} \right)
\end{align*}
with $\kappa \in \rr {2 \times 2}$ defined by \eqref{eq:kappa}. 
It follows that $\partial_x \partial_y k_b = \left( h \otimes h \right) \left( \left( 1 \otimes \delta_0 \right) \circ \kappa^{-1} \right)$ in $\cS'( \rr 2)$. 
If $\fy, \psi \in \cS(\ro)$ then
\begin{align*}
\left( k_{b'}, \fy \otimes \overline{\psi} \right)
& = \bE \left( (b',\fy) \overline{(b', \psi)}\right)
= \bE \left( (b,\fy') \overline{(b, \psi')}\right)
= \left( k_b, \fy' \otimes \overline{\psi'} \right)
= \left( \partial_x \partial_y k_b, \fy \otimes \overline{\psi} \right)
\end{align*}
and hence $k_{b'} = \left( h \otimes h \right) \left( \left( 1 \otimes \delta_0 \right) \circ \kappa^{-1} \right)$ in $\cS'( \rr 2)$.
We may conclude that $(\cK_{b'} \psi, \fy) = (h \psi, h \fy)_{L^2}$ for $\fy, \psi \in \cS(\ro)$, 
that is, $\cK_{b'} = h \cdot I$ is the identity operator times $h$ on $\cS(\ro)$. 
Comparing with Definition \ref{def:WhiteNoise} we may conclude that $b' \in \cL( \cS(\ro), L_0^2(\Omega))$
is white noise with unit power multiplied with $h$.

From \eqref{eq:covbrown} we infer $k_b \circ \kappa (x,y) = x - \frac12 |y|$ provided $x - \frac12 |y| \geqs 0$ and 
$k_b \circ \kappa (x,y) = 0$ otherwise. 
Thus from \eqref{eq:KernelWeylsymbol} it follows that the Weyl symbol of the covariance operator $\sigma_{b} (x,\xi) = 0$ if $x \leqs 0$, 
and more generally we have
\begin{equation*}
\sigma_{b} (x,\xi) 
= h(x) \int_{|y| \leqs 2 x} \left( x - \frac12 |y| \right) e^{- i y \xi} \dd y  
= 2 x^2 \left( \frac{\sin(x \xi)}{x \xi} \right)^2 h(x).
\end{equation*}
This formula is not new since it appears in \cite[Eq. (2.164)]{Flandrin2}.  

However, the covariance function of the Wigner distribution of Brownian motion does apparently not appear in the literature. 
We obtain a formula for the Weyl symbol of the covariance operator of the zero mean Wigner distribution $W_0(b)$ from \eqref{eq:WeylWignerCoord} as
\begin{align*}
& \sigma_W (x_1, x_2, \xi_1, \xi_2) 
= \sigma_b \left( x_1 - \frac12 \xi_2, x_2 + \frac12 \xi_1 \right) \sigma_b \left( x_1 + \frac12 \xi_2, x_2 - \frac12 \xi_1 \right) \\
& =  4 h( 2 x_1 - |\xi_2|) 
\left( x_1^2 - \frac14 \xi_2^2 \right)^2
\left( \frac{\sin\left( \left( x_1 - \frac12 \xi_2 \right) \left( x_2 + \frac12 \xi_1 \right) \right) }
{\left( x_1 - \frac12 \xi_2 \right) \left( x_2 + \frac12 \xi_1 \right) } 
\cdot \frac{ \sin\left( \left( x_1 + \frac12 \xi_2 \right) \left( x_2 - \frac12 \xi_1 \right) \right) }
{ \left( x_1 + \frac12 \xi_2 \right) \left( x_2 - \frac12 \xi_1 \right) } 
\right)^2 \\
& =  h( 2 x_1 - |\xi_2|) 
\left( 
\frac{ \cos\left( x_1 \xi_1 - x_2 \xi_2 \right)
- \cos\left( 2 x_1 x_2 - \frac12 \xi_1 \xi_2 \right) }
{x_2^2 - \frac14 \xi_1^2}
\right)^2
\end{align*}
where $x_1, x_2, \xi_1, \xi_2 \in \ro$.
Hence $\supp \sigma_W \subseteq \ro_+ \times \rr 3$.
We also observe that $\sigma_W (x_1, x_2, \xi_1, \xi_2) = 0$ when
$|x_1| = \frac12 |\xi_2|$.

\section{Application to non-negative operators in the Weyl calculus}\label{sec:NonnegWeylCalc}

In this section we deduce a consequence of Theorem \ref{thm:symbol} formulated for non-negative operators in the Weyl
calculus of pseudodifferential operators. 
Recall that if $a \in \cS'(\rr {2d})$ then $a^w(x,D): \cS(\rr d) \to \cS'(\rr d)$ is non-negative on $\cS(\rr d)$ 
if $( a^w(x,D) f, f ) \geqs 0$ for all $f \in \cS(\rr d)$.
This is abbreviated as $a^w(x,D) \geqs 0$ on $\cS(\rr d)$. 

\begin{prop}\label{prop:posWeyl}
Suppose $a \in \cS'(\rr {2d})$ and $a^w(x,D) \geqs 0$ on $\cS(\rr d)$.
If 
\begin{align*}
b (x_1, x_2, \xi_1, \xi_2) 
& = 
a \left( x_1 - \frac12 \xi_2, x_2 + \frac12 \xi_1 \right) a \left( x_1 + \frac12 \xi_2, x_2 - \frac12 \xi_1 \right) \in \cS'(\rr {4d}), \\
& \qquad \qquad x_1, x_2, \xi_1, \xi_2 \in \rr d, 
\end{align*}
then $b^w(x,D) \geqs 0$ on $\cS(\rr {2d} )$. 
\end{prop}

\begin{proof}
With $a^w = a^w(x,D)$ we have for all $f,g \in \cS(\rr d)$
\begin{equation*}
4 ( a^w f,g) = ( a^w(f+g), f+g ) - ( a^w(f-g), f-g ) + i \Big( ( a^w(f+ig), f+ig ) - (  a^w (f-ig), f-ig ) \Big). 
\end{equation*}
By the assumption $a^w (x,D) \geqs 0$ on $\cS(\rr d)$ it follows that 
\begin{equation}\label{eq:hermsym}
\overline{( a^w (x,D) g,f) } = ( a^w (x,D) f,g), \quad f,g \in \cS(\rr d). 
\end{equation}

Denote by $k =  (2 \pi)^{-\frac{d}{2}} ( \cF_2^{-1} a ) \circ \kappa^{-1} \in \cS'(\rr {2d})$, with $\kappa \in \rr {2d \times 2d}$ defined by \eqref{eq:kappa}, the Schwartz kernel of the operator $a^w(x,D)$, cf. \eqref{eq:KernelWeylsymbol}. 
The property \eqref{eq:hermsym} entails the Hermitian symmetry
\begin{equation}\label{eq:hermitiankernel}
\begin{aligned}
(\re \, k, \overline f \otimes g) & = (\re \, k, g \otimes \overline f), \\
(\im \, k, \overline f \otimes g) & = - (\im \, k, g \otimes \overline f), 
\end{aligned}
\end{equation}
and hence for any real-valued Schwartz functions $f,g \in \cS(\rr d,\ro)$ we have
\begin{equation}\label{eq:nonnegkernel}
\begin{aligned}
0 & \leqs \frac12 \left( k, ( f + i g) \otimes \overline{( f + i g)}\right) \\
& = \frac12 \left( \re \, k, f \otimes f + g \otimes g \right)
+ \frac12 \left( \im \, k, g \otimes f - f \otimes g \right).
\end{aligned}
\end{equation}

Let $n \in \no \setminus 0$, let $\{ \fy_j \}_{j=1}^n \subseteq \cS(\rr d, \ro)$ and define $\fy_x = \sum_{j=1}^n x_j \fy_j \in \cS(\rr d)$ for $x  = (x_1, \dots, x_n) \in \rr n$. 
If $x,y \in \rr n$ then 
\begin{equation}\label{eq:matrices}
\begin{aligned}
\left( \re \, k, \fy_x \otimes \fy_x \right) & = \left\langle \left( \re \, k, \fy_j \otimes \fy_k \right)_{jk} x, x \right\rangle = \la A_1 x, x \ra, \\
\left( \im \, k, \fy_x \otimes \fy_y \right) & = \left\langle \left( \im \, k, \fy_j \otimes \fy_k \right)_{jk} y, x \right\rangle = \la A_2 y, x \ra, 
\end{aligned}
\end{equation}
where  
$A_1 = \left( \re \, k, \fy_j \otimes \fy_k \right)_{jk} =  \left( \re \, k, \fy_j \otimes \fy_k \right)_{j,k=1}^n  \in \rr {n \times n}$
is a symmetric real matrix, and 
$A_2 = \left( \im \, k, \fy_j \otimes \fy_k \right)_{jk} \in \rr {n \times n}$ 
is an antisymmetric matrix. 
In fact these properties are consequences of \eqref{eq:hermitiankernel}. 

We obtain from \eqref{eq:nonnegkernel} with $f = \fy_x$ and $g = \fy_y$
and \eqref{eq:matrices}
\begin{equation}\label{eq:nonnegdef}
\begin{aligned}
0 \leqs \frac12 \Big( \left\langle A_1 x, x \right\rangle + \left\langle A_1 y, y \right\rangle 
+ \left\langle A_2 x, y \right\rangle - \left\langle A_2 y, x \right\rangle
\Big)
= \left\langle A (x,y), (x,y) \right\rangle
\end{aligned}
\end{equation}
where 
\begin{equation*}
A = \frac12
\left( 
\begin{array}{cc}
A_1 & - A_2 \\
A_2& A_1
\end{array}
\right) \in \rr {2n \times 2n}
\end{equation*}
is a symmetric matrix.
Since $x,y \in \rr n$ are arbitrary, it follows from \eqref{eq:nonnegdef} that 
$A$ is non-negative definite. 
Thus it a covariance matrix. 

From \cite[Chapter~3 \ \S 2.3]{Gelfand4} it now follows that there exists two real-valued Gaussian zero mean tempered $\gsp$s
$v_1, v_2 \in \cL ( \cS(\rr d,\ro) , L_0^2(\Omega) )$ such that the $2 n$-vector 
$(v_1, \fy_1), \cdots, (v_1, \fy_n), (v_2, \fy_1), \cdots, (v_2, \fy_n)$ in $L_0^2(\Omega)$ has covariance matrix equal to $A$. 
We infer for $\fy, \psi \in \cS(\rr d,\ro)$
\begin{align*}
& \bE \left( (v_j, \fy) (v_j, \psi) \right) = \frac12 (\re \, k, \fy \otimes \psi), \quad j = 1,2, \\
& \bE \left( (v_1, \fy) (v_2, \psi) \right) = - \frac12 (\im \, k, \fy \otimes \psi), 
\end{align*}
and these identities extend to complex-valued functions $\fy, \psi \in \cS(\rr d)$ as 
\begin{align*}
& \bE \left( (v_j, \fy) \overline{(v_j, \psi)} \right) = \frac12 (\re \, k, \fy \otimes \overline \psi), \quad j = 1,2, \\
& \bE \left( (v_1, \fy) \overline{(v_2, \psi)} \right) = - \frac12 (\im \, k, \fy \otimes \overline \psi).
\end{align*}

If we finally set $u = v_1 + i v_2$ then we obtain for $\fy, \psi \in \cS(\rr d)$
\begin{align*}
\bE \left( (u, \fy) \overline{(u, \psi)} \right) & = (k, \fy \otimes \overline \psi ), \\
\bE \left( (u, \fy) (u, \psi) \right) & \equiv 0. 
\end{align*}
This means that $u$ is Gaussian symmetric and $k = k_u$. 
From Theorem \ref{thm:symbol} and \eqref{eq:WeylWignerCoord} we deduce that 
$b^w(x,D)$ is the covariance operator of the zero mean Wigner distribution $W_0(u)$ of $u$. 
Hence $b^w(x,D) \geqs 0$ on $\cS(\rr {2d} )$.
\end{proof}

\section{Remarks on non-negativity of the Wigner spectrum}\label{sec:NonnegWignerSpectrum}

In the applied literature the Weyl symbol of the covariance operator of a $\gsp$ $u \in \cL ( \cS(\rr d ) , L_0^2(\Omega) )$
is called the Wigner spectrum \cite{Boashash1,Flandrin2,Flandrin3}.
Recall the formula \eqref{eq:EW} which says that the expected value of the Wigner distribution of $u$ equals
a positive constant times the Weyl symbol of the covariance operator:  
\begin{equation}\label{eq:EW2}
\bE W(u) = (2 \pi)^{-\frac{d}{2}} \sigma_u. 
\end{equation}
Modulo a positive multiplicative constant, the Wigner spectrum thus equals $\bE W(u)$.

It is interesting to compare $\bE W(u)$ with $W(f)$ defined by \eqref{eq:WignerSchwartz} with $g = f$ 
for a \emph{deterministic} function $f \in L^2(\rr d)$. 
In fact Wigner introduced $W(f)$ in 1932 as a candidate for a quantum mechanical probability density in phase space 
for a particle. Hudson's theorem \cite{Grochenig1,Hudson1,Janssen2} however shows that $W(f)$ cannot have probability density interpretation for all $f \in L^2(\rr d)$, since $W(f) \geqs 0$ everywhere in $T^* \rr d$ holds if and only if 
$f$ is a Gaussian of the form
\begin{equation*}
f(x) = e^{ - \la x, A x \ra  + \la b, x \ra + c}
\end{equation*}
where $A \in \GL(d,\co)$ has positive definite real part, $b \in \cc d$ and $c \in \co$. 

Discoveries by Flandrin \cite{Flandrin1,Flandrin2} and Janssen \cite{Janssen4} show that 
the corresponding problem for $\gsp$s admits a much larger class of solutions. 
This problem of finding $\gsp$s $u$ such that $\bE W(u) \geqs 0$ everywhere on $T^* \rr d$ is by 
\eqref{eq:EW2} equivalent to the problem of identifying non-negative 
linear operators $\cS(\rr d) \to \cS'(\rr d)$ whose Weyl symbols are nonnegative.

The next example taken from \cite[Chapter~3.3.3]{Flandrin2}
gives one class of stochastic processes on $\rr d$ for which the Wigner spectrum 
satisfies $\bE W(u) \geqs 0$ on $T^* \rr d$. 

\begin{example}\label{ex:nonnegweylsymbol}
Let $f \in L^2(\rr d) \setminus \{ 0 \}$ and define $u(x; y, \eta ) = M_\eta T_y  f(x)$
where $(y,\eta) \in \rr {2d}$ is a Gaussian random variable with probability density
\begin{equation*}
p(y,\eta) = (2 \pi)^{-d } ( a b)^{- \frac{d}{2}}  \exp\left( - \frac1{2a} |y|^2 - \frac1{2b} |\eta|^2 \right)
\end{equation*}
for $a, b > 0$. 
Writing $p(y,\eta) = p_1(y) p_2(\eta)$ 
with 
$p_1(y) = (2 \pi)^{- \frac{d}{2}} a^{- \frac{d}{2}}  e^{ - \frac1{2a} |y|^2 }$ and
$p_2(\eta) = (2 \pi)^{- \frac{d}{2}} b^{- \frac{d}{2}}  e^{ - \frac1{2b} |\eta|^2 }$
we have 
\begin{align*}
\bE u(x; \cdot) 
& = \iint_{\rr {2d}} p_1(y) p_2(\eta) e^{i \la \eta, x \ra} f(x-y) \dd y \dd \eta
= (p_1 * f)(x) (2 \pi)^{\frac{d}{2}} \wh{p}_2 (x) \\
& = (p_1 * f)(x) e^{- \frac{b}{2} |x|^2}
\end{align*}
and
\begin{equation*}
\bE | u(x; \cdot)|^2
= \iint_{\rr {2d}} p_1(y) p_2(\eta) |f(x-y)|^2 \dd y \dd \eta
= (p_1 * |f|^2)(x).
\end{equation*}
Thus $\rr d \ni x \mapsto u(x; \cdot)$ is a second order stochastic process
that belongs to the space $L^2( \rr d, L^2(\Omega) )$.

For $(y,\eta) \in \rr {2d}$ fixed we have by \cite[Proposition~4.3.2]{Grochenig1} 
\begin{equation*}
W( u(\cdot; y, \eta ) )(x,\xi) = W(f) (x-y, \xi-\eta)
\end{equation*}
which gives
\begin{equation*}
\bE W(u)(x,\xi) = \iint_{\rr {2d}} p(y,\eta) W(f) (x-y, \xi-\eta) \dd y \dd \eta
= p * W(f) (x,\xi). 
\end{equation*}
Now it follows from \cite[Theorem~4.4.4]{Grochenig1}
that $\bE W(u)(x,\xi) \geqs 0$ for all $(x,\xi) \in T^* \rr d$ provided $a b \geqs \frac14$. 
This gives a large class of stochastic processes such that $\bE W(u) \geqs 0$
on $T^* \rr d$. 

\end{example}

Under the restriction $\sigma_u \in L^2(\rr {2d})$ Janssen \cite{Janssen3} works out many
necessary conditions for $\sigma_u$ to be the Weyl symbol of a nonnegative operator. 
The assumption $\sigma_u \in L^2(\rr {2d})$ is equivalent to the assumption Hilbert--Schmidt for the covariance operator. 
In particular there are lower bounds on measures of spread of $\sigma_u$ that are close in spirit to uncertainty
principles.

\section*{Acknowledgement}

The author is grateful to Hans Feichtinger for helpful comments and for providing references.
The author is a member of Gruppo Nazionale per l’Analisi Matematica, la Probabilit\`a e le loro Applicazioni (GNAMPA) -- Istituto Nazionale di Alta Matematica
(INdAM).


\end{document}